\documentclass[10pt,a4paper]{article}
\usepackage{amssymb,amsmath,ae,amsthm,graphicx}%
\newcommand{\cnstKO}{K_{1,\,\delta}}
\newcommand{\cnstKT}{K_{2,\,\delta}}
\newcommand{\cnstKi}{K_{i,\,\delta}}

\newcommand{\eps}{\varepsilon}
\newcommand{\ignore}[1]{}
\newtheorem{prop}{Proposition}[section]
\newtheorem{lemma}[prop]{Lemma}
\newtheorem{remark}[prop]{Remark}
\newtheorem{theorem}[prop]{Theorem}
\newtheorem{corollary}[prop]{Corollary}
\newtheorem{definition}[prop]{Definition}
\newtheorem{assumption}[prop]{Assumption}

\newcommand{\fspace}{Z}
\newcommand{\fnorm}[1]{{\|{#1}\|}_{\!\fspace}}
\newcommand{\tnorm}[1]{{\lfloor{#1}\rfloor}_{\!\fspace}}
\newcommand{\snorm}[1]{\lceil{#1}\rceil_{\!\fspace}}

\newcommand{\VarDer}[2]{{\partial_{#2\,}}{#1}}
\newcommand{\eeq}{=}
\newcommand{\deq}{\,:=}
\newcommand{\teps}{\widetilde{\eps}}

\newcommand{\DO}[1]{{O\at{#1}}}
\newcommand{\Do}[1]{{o\at{#1}}}

\newcommand{\nDo}[1]{{o\nat{#1}}}

\newcommand{\Nset}{\mathbb{N}}

\newcommand{\app}{{\rm app}}
\newcommand{\res}{{\rm res}}

\newcommand{\al}{{\alpha}}

\newcommand{\ga}{{\gamma}}

\newcommand{\si}{{\sigma}}

\newcommand{\pair}[2]{{\left({#1},\,{#2}\right)}}
\newcommand{\npair}[2]{{({#1},\,{#2})}}

\newcommand{\at}[1]{{\left({#1}\right)}}
\newcommand{\nat}[1]{{({#1})}}
\newcommand{\bat}[1]{{\big(#1\big)}}
\newcommand{\Bat}[1]{{\Big(#1\Big)}}
\newcommand{\triple}[3]{{\left({#1},\,{#2},\,{#3}\right)}}

\newcommand{\cinterval}[2]{[#1,\,#2]}%
\newcommand{\cointerval}[2]{[#1,\,#2)}%
\newcommand{\ccinterval}[2]{[#1,\,#2]}%

\newcommand{\bigpar}{\par$\;$\par\noindent}

\newcommand{\abs}[1]{\left|{#1}\right|}
\newcommand{\babs}[1]{\big|{#1}\big|}
\newcommand{\Babs}[1]{\Big|{#1}\Big|}

\renewcommand{\d}{\mathrm{d}}

\newcommand{\calY}{\mathcal{Y}}
\newcommand{\calZ}{\mathcal{Z}}

\newcommand{\mhdraft}{false}
\begin{document}
\title{Self-similar solutions for the {LSW} model with encounters}%
\date{\today}
\author{%
M. Herrmann%
\footnote{Institut f\"{u}r Mathematik, Humboldt-Universit\"{a}t zu
Berlin, Unter den Linden 6, 10099 Berlin, Germany},
B. Niethammer%
\footnote{Mathematical Institute, University of Oxford, 24-29 St.
Giles, Oxford, OX1 3LB, England},
and
J. J. L. Vel\'{a}zquez\footnote{Departamento de Matem\'{a}tica
Aplicada, Facultad de Matem\'{a}ticas, Universidad Complutense,
Madrid 28040, Spain}.
}%
\maketitle
\begin{abstract}
The LSW model with encounters has been suggested by Lifshitz and
Slyozov as a regularization of their classical mean-field model for
domain coarsening to obtain universal self-similar long-time
behavior. We rigorously establish that an exponentially decaying
self-similar solution to this model exist, and show that this
solutions is isolated in a certain function space. Our proof relies
on setting up a suitable fixed-point problem in an appropriate
function space and careful asymptotic estimates of the solution to a
corresponding homogeneous problem.
\end{abstract}
\quad\newline\noindent%
{keywords:} %
\emph{coarsening with encounters}, %
\emph{self-similar solutions},
\emph{kinetics of phase transitions}%

%
%
%---------------------------------------------------------------------
\section{Introduction}
%---------------------------------------------------------------------
%
%
The classical mean-field theory by Lifshitz and Slyozov \cite{LS1}
and Wagner \cite{Wa1} describes domain coarsening of a dilute system
of particles which interact by diffusional mass transfer to reduce
their total interfacial area. It is based on the assumption that
particles interact only via a common mean-field $\theta=\theta(t)$
which yields a nonlocal transport equation for the number density
$f=f(v,t)$ of particles with volume $v$. It is given by
\begin{equation}\label{LSW1}
\frac{\partial f}{\partial t} + \frac{\partial }{\partial v} \left(
\left( -1+\theta\left(  t\right)  v^{1/3}\right)  f\right)  =0,
\qquad v >0,\quad t>0,
\end{equation}
where $\theta(t)$ is determined by the constraint that the total
volume of the particles is preserved in time, i.e.
\begin{equation}\label{LSW2}
\int\limits_0^{\infty} v f(v,t)\,dv = \rho\,.
\end{equation}
This implies that
\begin{equation}
\theta\left(  t\right)     =\frac{1}{\left\langle
v^{1/3}\right\rangle
}=\frac{\int\limits_{0}^{\infty}f\left(  v,t\right)  dv}{\int\limits_{0}^{\infty}%
v^{1/3}f\left(  v,t\right)  dv}\label{S1E1h},
\end{equation}
where
\begin{math}
\left\langle
v^{k}\right\rangle\deq{}m_{k}\deq\int_{0}^{\infty}v^{k}f\left(
v,t\right) dv
%\label{S1E2b}%
\end{math} for $k>0$.
It is observed in experiments that  coarsening systems display
statistical self-similarity over long times, that is the number
density converges towards a unique self-similar form. Indeed, also
the mean-field model \eqref{LSW1}-\eqref{LSW2} has a scale
invariance, which suggests that typical particle volumes grow
proportional to $t$. Going over to self-similar variables one easily
establishes that there exists a whole one-parameter family of
self-similar solutions. All of the members of this family have
compact support and can be characterized by their behavior near the
end of their support: One is infinitely smooth, the others behave
like  a power law. It has been established in \cite{NP2} (cf. also
\cite{GMS} for asymptotics and numerical simulations, \cite{CP1}
results on a simplified problem and \cite{NV52} for some
refinements), that the long-time behavior of solutions to
\eqref{LSW1}-\eqref{LSW2} depends sensitively on the initial data,
more precisely on their behavior near the end of their support.
Roughly speaking, the solution converges to the self-similar
solution which behaves as a power law of power $p< \infty$ if and
only if the data are regularly varying with power $p$ at the end of
their support. The domain of attraction of the infinitely smooth
solution is characterized by a more involved condition \cite{NV52},
which we do not give here since it is not  relevant for the
forthcoming analysis.
\bigpar
This weak selection of self-similar asymptotic states reflects a
degeneracy in the mean-field model which is generally believed to be
due to the fact that the model is valid only in the regime of
vanishing volume fraction of particles \cite{Ni1}. Some effort has
been made to derive corrections to the classical mean-field model in
order to reflect the effect of positive volume fraction, such as
screening induced fluctuations \cite{Ma1,NV6,Ni6}, or to take
nucleation into account \cite{FaNe06a,Me1,Ve1}. A different approach
was already suggested by Lifshitz and Slyozov \cite{LS1} which is to
take the occasional merging of particles (``encounters'') into
account. This leads to the equation
\begin{equation}
\frac{\partial f}{\partial t}+\frac{\partial}{\partial v}\left(
\left( -1+\theta\left(  t\right)  v^{1/3}\right)  f\right) =J[f],
\label{S1E1}
\end{equation}
where $J[f]$ is a typical coagulation term, given by
\begin{equation*}
J[f] =\frac{1}{2}\int\limits_{0}^{v}w\left(  v-v^{\prime}%
,v^{\prime}\right)  f\left(  v-v^{\prime},t\right)  f\left(
v^{\prime },t\right)  dv^{\prime} \,-\, \int\limits_0^{\infty}
w(v,v^{\prime}) f(v^{\prime},t) f(v,t)\,dv^{\prime},
%\label{S1E2}\,,
\end{equation*}
with a suitable rate kernel $w$ specified below.
Volume conservation \eqref{LSW2} should still be valid, and since
\begin{equation*}
\int\limits_{0}^{\infty}vJ[f]\left(  v,t\right)  dv=0 %
%\label{S1E3}%
\end{equation*}
this requires that $\theta$ is again given by \eqref{S1E1h}.
\par
It remains to specify the rate kernel
 $w\left(  v,v^{\prime}\right)  $ which Lifshitz and Slyozov assume
to be dimensionless with respect to rescalings of $v,v^{\prime}$ and
to be additive for large values of $v$ and $v^{\prime}$. For
simplicity we assume -- just as in \cite{LS1} -- that
\begin{equation}
w\left(  v,v^{\prime}\right)  =\left(  \frac{\left\langle
v\right\rangle +\left\langle v^{\prime}\right\rangle }{2}\right)
^{-1}\left(  v+v^{\prime }\right)  =\left\langle v\right\rangle
^{-1}\left(  v+v^{\prime}\right)\,,
\label{S1E3a}%
\end{equation}
that is we obtain a coagulation term with the so-called ``additive
kernel''. Well-posedness of \eqref{S1E1} with this kernel has been
established in \cite{Laur3}.
\par
As explained before, the model \eqref{S1E1}, \eqref{LSW2}  is
relevant in the regime that the volume fraction covered by the
particles is small and hence we assume that
\begin{equation}
\int\limits vf\left(  v,t\right)  dv=\varepsilon \ll 1 \,.\label{S1E4}%
\end{equation}
The system (\ref{S1E1})-(\ref{S1E4}) can now be written in
self-similar variables
\begin{align*}
f\left(  v,t\right)=\frac{\varepsilon}{t^{2}}\Phi\left(  \frac{v}%
{t},\log\left(  t\right)  \right)
,\quad\quad%
z=\frac{v}{t},\quad\quad\tau =\log\left(  t\right)
,\quad\quad%
\theta\left(  t\right)=\frac{\lambda\left(  \tau\right) }{t^{1/3}}
\end{align*}
as
\begin{equation}
\Phi_{\tau}-z\Phi_{z}-2\Phi+\frac{\partial}{\partial z}\left( \left(
-1+\lambda\left(  \tau\right)  z^{1/3}\right)  \Phi\right)
=\varepsilon J[\Phi]\left(  z,\tau\right) \,,\label{S1E7}
\end{equation}
\begin{equation}
\int\limits_0^{\infty} z \Phi(z,\tau)\,dz =1\,, \label{S1E8}
\end{equation}
where
\begin{equation}
\notag%
J[\Phi]  \left(  z,\tau\right)     =\frac{1}{2}\int\limits_{0}^{z}%
z\Phi\left(  z-z^{\prime},\tau\right)  \Phi\left(
z^{\prime},\tau\right) dz^{\prime}
  -\Phi\left(  z,\tau\right)  \int\limits_{0}^{\infty}\left(  z+z^{\prime}\right)
\Phi\left(  z^{\prime},\tau\right)  dz^{\prime}
%\label{S1E7a}\\
\end{equation}
and
\begin{equation}
\notag%
\lambda\left(  \tau\right) =\frac{\int\limits_{0}^{\infty}\Phi\left(
z,t\right) dz}{\int\limits_{0}^{\infty}z^{1/3}\Phi\left(  z,t\right)
dz}
%\label{S1E7b}
\end{equation}
Our  goal in this paper is to study stationary solutions of
\eqref{S1E7}-\eqref{S1E8} in the regime of small $\eps$. We notice
first that the convolution term on the right hand side of
\eqref{S1E7} enforces that any solution must have infinite support.
We also expect that for small $\eps>0$ the solution should be close
-- in an appropriate sense -- to one of the self-similar solutions
of the LSW model, that is \eqref{S1E7} with $\eps=0$. It can be
verified by a stability argument that the only solution of the LSW
model for which this is possible is the smooth one which has the
largest support.
\par
Indeed, we obtain as our main result, that for any given
sufficiently small $\eps>0$ there exists an exponentially decaying
stationary solution to \eqref{S1E7}-\eqref{S1E8}. Moreover, we show
this solution to be isolated, i.e., there is no further solution
with exponential tail in a sufficiently small neighborhood of the
LSW solution. We do not believe, that there exist other
exponentially decaying solutions, but our proof does not exclude
that. However, we conjecture that there are algebraically decaying
stationary solutions as well. We are not yet able to establish a
corresponding result for the model discussed in the present paper,
but can prove this for a simplified model (see \cite{HLN1}).
\bigpar
For the pure coagulation equation, that is \eqref{S1E7} without the
drift term, an exponentially decaying stationary solution exists
only for $\eps=1/2$. For every smaller $\eps$ there exists a
stationary solution with algebraic decay. The domain of attraction
of these self-similar solutions has been completely characterized in
\cite{MP1}, and can also be related to the regular variation of
certain moments of the initial data.
\par
However, the situation here is somewhat different. While the
behavior for large volumes $v$ is determined by the coagulation
term, the tail introduced by the coagulation term is very small, and
the equation behaves - at least in the regime in which we are
working - as the LSW model with a small perturbation. Our analysis
reflects this fact, since we also treat the coagulation term as a
perturbation.
%
%
%
%
%---------------------------------------------------------------------
\section{Statement of the fixed point problem}
\label{sec:Problem}
%---------------------------------------------------------------------
%
%
In this section we set up a suitable fixed point problem for
the construction of stationary solutions to
\eqref{S1E7}-\eqref{S1E8}. These solve
\begin{align}
\label{S2E1}%
-z\frac{\partial\Phi}{\partial z}-2\Phi+\frac{\partial}{\partial
z}\left( \left(  -1+\lambda z^{1/3}\right)  \Phi\right)   &
=\varepsilon J[\Phi]\left( z\right),\quad\quad
\int\limits_{0}^{\infty}z\Phi\left(  z\right) dz =1,
\quad\quad\Phi\left(  z\right)   \geq0,%
\end{align}
with $z>0$ and
\begin{equation*}
\lambda=\frac{\int\limits_{0}^{\infty}\Phi\left( z\right)
dz}{\int\limits_{0}^{\infty }z^{1/3}\Phi\left(  z\right) dz}.
%\label{S2E3}%
\end{equation*}
In the LSW limit $\eps=0$ there exists a family of solutions with
compact support, which can be parametrized by the mean field
$\lambda\in\left[ 3\left(  \frac{1}{2}\right) ^{2/3},\infty\right)
.$ The self-similar solution with the largest support, which is
$[0,1/2]$, is exponentially smooth and is given by
\begin{equation*}
%\label{S2E6a}
\Phi_{LSW}(z)= %
\left\{%
\begin{array}{lcl}
\displaystyle%
C\exp\at{- \int\limits_0^z
\frac{2-\tfrac{1}{3}\lambda_{LSW}\xi^{-2/3}}{\xi + 1 -
\lambda_{LSW}\xi^{1/3}} \,d\xi }
&\text{for}&z \in [0,\tfrac{1}{2}],%
\\%
0&\text{for}&{z > \tfrac{1}{2}},
\end{array}
\right.
\end{equation*}
with
\begin{equation*}
\notag%\label{S2E6}
\lambda_{LSW}:=3\left(  \frac{1}{2}\right) ^{2/3}
\end{equation*}
and $C$ is a normalization constant chosen such that
$\int_0^{\infty} z \Phi_{LSW}\,dz= 1$. We denote from now on this
solution by $\Phi_{LSW}$. As discussed above  there
 are several physical and mathematical arguments supporting the fact that
such a solution is the only stable one under perturbations of the
model.
\bigpar
The main goal of this paper is to show the following result.
\begin{theorem}
\label{Intro.MainResult} For any sufficiently small
$\lambda_{LSW}-\lambda$ there exists a choice for $\eps$ such that
there exists an exponentially decaying solution to \eqref{S2E1}.
\end{theorem}
The key idea for proving this theorem is to reduce the problem to a
standard fixed point problem assuming that \eqref{S2E1} is a small
perturbation of $\Phi_{LSW}$.
%
%
%
%---------------------------------------------------------------------
\subsubsection*{Formal asymptotics as $\varepsilon\rightarrow0$}
%---------------------------------------------------------------------
%
%
The formal asymptotics of the solution of \eqref{S2E1} whose
existence we prove in this paper was obtained in \cite{LS1}. We
recall it here for convenience. Such a solution $\Phi$ is expected
to be close to the solution $\Phi_{LSW}$ as
$\varepsilon\rightarrow0.$ Notice, however that $\Phi_{LSW}$
vanishes for $z\geq\frac{1}{2}.$ Therefore, in order to approximate
$\Phi$ for $z\geq\frac{1}{2}$ Lifshitz and Slyozov
 approximate (\ref{S2E1}) by means of%
\begin{equation}
-z\frac{\partial\Phi}{\partial z}-2\Phi+\frac{\partial}{\partial
z}\left( \left(  -1+\lambda z^{1/3}\right)  \Phi\right) =\varepsilon
J[\Phi_{LSW}]\left(  z\right)  \label{T1}%
\end{equation}There exists a unique solution of (\ref{T1}) which vanishes for $z
\geq 1$. Such a function is of order $\varepsilon$ in the interval
$\left( \frac{1}{2},1\right)  .$ However, there is a boundary layer
in the region $z\approx\frac{1}{2}$ for $\lambda$ close to
$\lambda_{LSW}$ where the function $\Phi$ experiences an abrupt
change. Adjusting the value of $\lambda$ in a suitable manner it is
possible to obtain $\Phi$ which is of order one for
$z<\frac{1}{2}.$ A careful analysis shows that $\lambda$ must be chosen as%
\begin{equation}\label{T1b}
\lambda_{LSW}-\lambda\sim\frac{3\pi^{2}}{\left(  2\right)
^{2/3}}\frac
{1}{\left(  \log\left(  \varepsilon\right)  \right)  ^{2}}%
\end{equation}
as $\varepsilon\rightarrow0$. This scaling law was already derived
in \cite{LS1}, and is in accordance with our results. Notice that
the smallness of $\Phi$ for $z\geq\frac{1}{2}$ implies that most of
the volume of the particles is in the region $z<\frac{1}{2}.$
\par%
In order to approximate $\Phi$ for $z\geq1$ we would need to use the
values of the function $\Phi$ for $z\in\cinterval{\tfrac{1}{2}}{1}$
obtained by means of (\ref{T1}). Therefore, $\varepsilon J[\Phi]$
becomes of order $\varepsilon^{2}$ for
$z\in\cinterval{1}{\tfrac{3}{2}}$ and the contribution of this
region can be expected to be negligible compared to that of the
interval $\cinterval{\tfrac{1}{2}}{1}$. A similar argument indicates
that the contributions to $\Phi$ due to the operator $\varepsilon
J\left( \Phi\right)  $ for $z>\frac{3}{2}$ can be ignored. This
procedure can be iterated to obtain in the limit a solution to
\eqref{S1E7} which decays exponentially fast at infinity. What
remains to be established is that such a procedure indeed leads to a
converging sequence of solutions. A rigorous proof could be based on
such a procedure; we proceed, however, in a slightly different
manner.
\bigpar
Before we continue we briefly comment on \eqref{T1b}, which give the
deviation of the mean-field from the value of the LSW model. This
quantity is of particular interest, since its inverse is a measure
for the coarsening rate, which is one of the key quantities in the
study of coarsening systems. Equation \eqref{T1b} predicts a much
larger deviation than the ones obtained from other corrections to
the LSW models. For example, one model which takes the effect of
fluctuations into account \cite{NV6} predicts a deviation of order
$O(\eps^{1/4})$. The large deviation predicted by \eqref{T1} can be
attributed to the fact that all particles contribute to the
coagulation term and suggests, that encounters are more relevant in
the self-similar regimes than fluctuations. We refer to \cite{Ni6}
for a more extensive discussion of these issues.
%
%---------------------------------------------------------------------
\subsubsection*{Derivation of a fixed point problem}
%---------------------------------------------------------------------
%
%
We now transform \eqref{S2E1} with the choice of kernel
(\ref{S1E3a}) into a fixed point problem. To this end we write our
equation as follows
\begin{equation}
-z\frac{\partial\Phi}{\partial z}-2\Phi+\frac{\partial}{\partial
z}\left( \left(  -1+\lambda z^{1/3}\right)  \Phi\right)
=\varepsilon\left[  \frac {z}{2}\int\limits_{0}^{v}\Phi\left(
z-z^{\prime}\right)  \Phi\left(  z^{\prime }\right)
dz^{\prime}-\Phi\left(  z,t\right)  -m_{0}z\Phi\left(  z\right)
\right]  \label{S4E1}%
\end{equation}
with
\begin{equation}
1=\int\limits_{0}^{\infty}z\Phi\at{z}\,d{z} ,
\quad\quad\quad%
m_{0}=\int\limits_{0}^{\infty}\Phi\at{z}\,dz.
\label{S4E2}%
\end{equation}
It would be natural to proceed as follows: For each given value of
$\varepsilon$ we select $m_{0}$ and $\lambda$ in order to satisfy
\eqref{S4E2}. However, it turns out to be more convenient to fix
$\lambda$ and then select $\eps$ and $m_{0}$ such that \eqref{S4E2}
is satisfied. The reason is that our argument requires to
differentiate the function $\psi$ defined below with respect to
either $\lambda$ or $\eps$, but it is easier to control the
derivatives with respect to $\eps$.
\bigpar
In the following we always consider $\lambda< \lambda_{LSW}$
and write
\begin{equation*}
\delta\deq\lambda_{LSW} -
\lambda>0,\quad\quad\quad\quad\tilde{\eps}\deq\eps{m_0}>0.
%\label{S2E12b}%
\end{equation*}
An important role in the fixed point argument is played by the functions
$z\mapsto\psi\at{z;\eps,\teps,\delta}$, which are defined as
solution to the following homogeneous problem
\begin{equation}
-\at{1+z-\at{\lambda_{LSW}-\delta}
z^{1/3}}\psi^{\prime}=\at{2-\tfrac{1}{3}\at{\lambda_{LSW}-\delta}
z^{-2/3}-\teps{z}-\eps)}\psi.
\label{S2E10}%.
\end{equation}
Each of these functions $\psi$ is uniquely determined up to a
constant to be fixed later. Notice that for $\delta>0$ the
function $\psi\at{z;\eps,\teps,\delta}$ is defined for all $z\geq0$.
 If $\delta=0$, however, the function $\psi\at{z;\eps,\teps,0}$ is defined only in the set
$z>\frac{1}{2},$ and it becomes singular as
$z\rightarrow\big(\frac{1}{2}\big)^{+}.$ Therefore the function
$\psi$ changes abruptly in a neighborhood of $z=\frac{1}{2}$ for
$\lambda$ close to $\lambda_{LSW}$. More precisely, if $\psi$ takes
values of order one for $z<\frac{1}{2}$,  then it is of order
$\exp\nat{-c/\sqrt\delta}$ for $z>\frac{1}{2}$, and this transition
layer causes most of the technical difficulties.
\bigpar%
We can now transform (\ref{S4E1}) into a fixed point problem for an
integral operator. Indeed, using Variation of Constants, and
assuming that $\Phi\left(  z\right)  $ decreases sufficiently fast
to provide the integrability required in the different formulas, we
obtain that each solution to \eqref{S4E1} satisfies
\begin{equation}
\label{S4E3}%
\Phi\at{z}=\eps\int\limits\limits_{z}^{\infty} \frac{\xi
}{\at{\xi+1-\at{\lambda_{LSW}-\delta}\xi^{1/3}}}%
\frac{\psi\at{z;\eps,\teps,\delta}}{\psi\at{\xi;\eps,\teps,\delta}}\at{\Phi*\Phi}\at\xi\,
d\xi=:I\left[ \Phi;\eps,\teps,\delta\right] \left( z\right),
\end{equation}
where the symmetric convolution operator $*$ is defined by
\begin{equation}
\label{S7E4a}
\at{\Phi_1 * \Phi_2}(z) =%
\tfrac{1}{2}\int\limits\limits_0^z \Phi_1(z-y)\Phi_2(y)\,dy= %
\tfrac{1}{2}\int\limits\limits_0^z\Phi_2(z-y)\Phi_1(y)\,dy.
\end{equation}
However, the values of the parameters $\pair{\eps}{\teps}$ cannot be
chosen arbitrarily but must be determined by the \emph{compatibility
conditions}
\begin{equation}
\eps\int\limits\limits_{0}^{\infty}zI\left[
\Phi;\eps,\teps,\delta\right]
dz\eeq\eps,\quad\quad\quad\eps\int\limits\limits_{0}^{\infty}I\left[
\Phi;\eps,\teps,\delta\right]  dz\eeq\teps\,.
\label{eq:Norm.Cond}%
\end{equation}
Notice that the operator $I\left[\Phi;\eps,\teps,\delta\right]$ maps
the cone of nonnegative functions $\Phi$ into itself, and this
implies that each solution to \eqref{S4E3} is nonnegative.
%
%
%---------------------------------------------------------------------
\subsubsection*{Main results and outline of the proofs}
%\label{sec:MRandO}
%---------------------------------------------------------------------
%
%
We introduce the following function space $Z$ of exponentially
decaying functions. For arbitrary but fixed constants $\beta_1>0$
and $\beta_2>1$ let $\fspace\deq\{\Phi\;:\;\fnorm{\Phi}<\infty\}$
with
\begin{align}
\label{Def:Spaces}
\fnorm{\Phi}\deq\snorm{\Phi}+\tnorm{\Phi},\qquad
\snorm{\Phi}\deq\sup\limits_{0\leq{z}\leq1}\abs{\Phi\at{z}},
\qquad\tnorm{\Phi}\deq
\sup\limits_{z\geq1}\abs{\Phi\at{z}\exp\at{\beta_1{}z}z^{\beta_2}}.
\end{align}
Below in Section \ref{sect:AuxResults} we prove that $\Phi\in{Z}$
implies $\Phi*\Phi\in{Z}$. The particular choice of the parameters
$\beta_1$ and $\beta_2$ affects our smallness assumptions for the
parameter $\delta$: The larger $\beta_1$ and $\beta_2$ are the
smaller $\delta$ must be chosen, and the faster the solution will
decay. We come back to this issue at the end of the paper, cf.
Remark \ref{Rem.Prms.Beta}.
\bigpar
Our (local) existence and uniqueness results relies on the following smallness assumptions concerning $\delta$,
$\eps$, $\teps$ and $\Phi$.
\begin{assumption}
\label{Ass.Prms.1} Suppose that
\begin{enumerate}
\item $\delta$ is sufficiently small,
\item both $\eps$ and $\teps$ are of order $\nDo{\sqrt{\delta}}$,
\item $\Phi$ is sufficiently close to $\Phi_{LSW}$, in the sense
that $\fnorm{\Phi-\Phi_{LSW}}$ is small.
\end{enumerate}
\end{assumption}
Our first main result guarantees that we can choose the
parameters $\eps$ and $\teps$ appropriately.
\begin{theorem}
\label{MainTheo1}
Under Assumption \ref{Ass.Prms.1} we can solve \eqref{eq:Norm.Cond},
i.e., for each $\Phi$ there exists a unique choice of
$\pair{\eps}{\teps}$ such that the compatibility conditions are
satisfied. This solution belongs to
\begin{equation*} %\label{S9E2a}
U_\delta=\left\{\pair{\eps}{\teps}\;:\;\tfrac{1}{2}\epsilon_\delta\leq\eps\leq2\epsilon_\delta,\;
\tfrac{1}{2}\widetilde{\epsilon}_\delta\leq\teps\leq2\widetilde{\epsilon}_\delta\right\},
\end{equation*}
where $\epsilon_\delta\sim\exp\nat{-c/\sqrt\delta}$ and
$\widetilde{\epsilon}_\delta\sim\nat{-c/\sqrt\delta}$ will be identified in
Equation \eqref{Eqn:Eps.Asymp} below.
\end{theorem}
The solution from Theorem \ref{MainTheo1} depends naturally on the
function $h=h[\Phi]=\Phi*\Phi$, and is denoted by
$\pair{\eps[h;\delta]}{\teps[h;\delta]}$. In a second step we define
an operator $\bar{I}_\delta[\Phi]$ via
\begin{align}
\label{Def:MainFPOp}
\bar{I}_\delta[\Phi]:=
I\left[\Phi;\,\eps\big[h[\Phi];\delta\big],\,\teps\big[h[\Phi];\delta\big],\,\delta\right]%
\end{align}
with $I$ as in \eqref{S4E3}, and show that for sufficiently small
$\delta$ there exists a corresponding fixed point.
\begin{theorem}
\label{MainTheo2} %
Under Assumption \ref{Ass.Prms.1} there exists a nonnegative
solution to $\Phi=\bar{I}_\delta[\Phi]$ that is isolated in the
space $\fspace$.
\end{theorem}
In order to prove Theorem \ref{MainTheo1} we rewrite the
compatibility conditions as a fixed point equation for
$\pair{\eps}{\teps}$ with parameters $\delta$ and $h$. This reads
\begin{align}
\label{Def.FP.Prms}
g_{1}\left[h; \eps,\teps,\delta\right]\eeq\eps,\qquad %
g_{2}\left[h; \eps,\teps,\delta\right]\eeq\teps,
\end{align}
where
\begin{align}
\label{Def.FP.Prms.Def}
\begin{split}
g_{i}[h;\eps,{\teps},\delta]&\deq
\eps^2\int\limits_{0}^{\infty}%
\frac{\xi\,h\at\xi}{\at{\xi+1-\at{\lambda_{LSW}-\delta}\xi^{1/3}}}
\int\limits_{0}^{\xi}\ga_i\at{z}
\frac{\psi\at{z;\eps,{\teps},\delta}}{\psi\at{\xi;\eps,{\teps},\delta}}
dzd\xi,%
\end{split}
\end{align}
with $\ga_1\at{z}=z$ and $\ga_2\at{z}=1$.
\bigpar
For fixed $h$ the integrals $g_1$ and $g_2$ depend extremely
sensitive on $\eps$, $\teps$, and $\delta$. Therefore, the crucial
part in our analysis are the following asymptotic expressions for
$g_1$, $g_2$ and their derivatives that we derive within Section
\ref{sect:FP.Prms}.
\begin{prop}
\label{Intro.Prop1} %
Assumption \ref{Ass.Prms.1} implies
\begin{equation*}
\abs{g_{1}[h;\eps,\teps,\delta]- \epsilon_\delta}=
%\Do{1}\epsilon_\delta
\Do{\epsilon_\delta}
,\quad\quad\quad%
\abs{g_{2}[h;\eps,\teps,\delta]- \widetilde{\epsilon}_\delta}=
%\Do{1}\widetilde{\epsilon}_\delta,
\Do{\widetilde{\epsilon}_\delta}
\end{equation*}
as well as
\begin{equation*}
\abs{\eps\VarDer{g_i}{\eps}-2g_i}\leq\Do{g_i},\quad\quad\quad
\abs{\teps\VarDer{g_i}{\teps}}\leq\Do{g_i},
\end{equation*}
for $i=1,2$.\end{prop}
Exploiting these estimates we prove Theorem \ref{MainTheo1} by
means of elementary analysis, see Section \ref{sect:FP.Prms}, and as
a consequence we derive in Section \ref{sect:FP.Prms} the following
result, which in turn implies Theorem \ref{sect:FP.Phi}.
\begin{prop}
\label{Intro.Prop2} %
Under Assumption \ref{Ass.Prms.1} there exists a small ball around
$\Phi_{LSW}$ in the space $Z$ such that the operator
$\bar{I}_\delta[\Phi]$ is a contraction on this ball.
\end{prop}
We proceed with some comments concerning the uniqueness of
solutions. Proposition \ref{Intro.Prop2} provides a local uniqueness
result in the function space $\fspace$. Moreover, since we can
choose the decay parameters $\beta_1>0$ and $\beta_2>1$ arbitrarily,
compare the discussion in Remark \ref{Rem.Prms.Beta}, we finally
obtain local uniqueness in the space of all exponential decaying
solutions. However, this does exclude neither the existence of
further exponentially decaying solutions being sufficiently far away
from $\Phi_{LSW}$, nor the existence of algebraically decaying
solutions.
%
%
%
%---------------------------------------------------------------------
\section{Proofs } \label{sec:MT.g}
%---------------------------------------------------------------------
In what follows $c$ and $C$ are small and large, respectively,
positive constants which are independent of $\delta$. Moreover,
$\Do{1}$ denotes a number that converges to $0$ as $\delta\to0$,
where this convergence is always uniform with respect to all other
quantities under consideration.
\bigpar%
For the subsequent considerations the following notations are
useful. For $\delta>0$ let
\begin{align}
\label{Def.Elem.Funct}%
\begin{split}
a_\delta\at{z}
\deq\frac{1}{1+z-\lambda_{LSW}{z}^{1/3}+\delta{z}^{1/3}},\quad\quad\quad
b_\delta\at{z}\deq2-\tfrac{1}{3}\lambda_{LSW}{z}^{-2/3}+\tfrac{1}{3}\delta{z}^{-2/3},
\end{split}
\end{align}
compare Figure \ref{Fig1}, so that by definition the function
$z\mapsto\psi\at{z;\eps,\teps,\delta}$ solves the homogenous
equation
\begin{figure}[ht!]%
\centering{%
\includegraphics[width=0.35\textwidth,draft=\mhdraft]{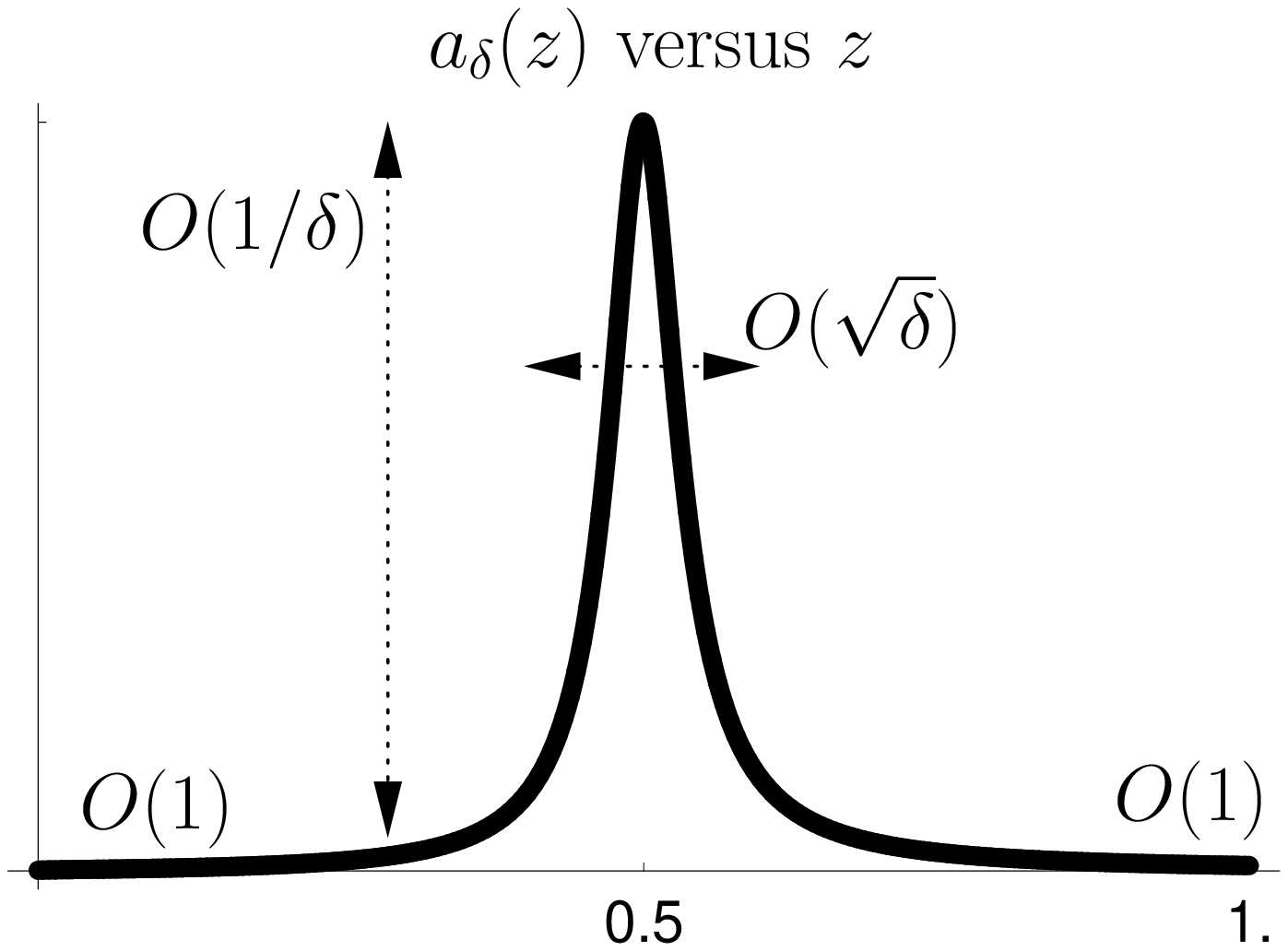}%
\hspace{0.04\textwidth}%
\includegraphics[width=0.35\textwidth,draft=\mhdraft]{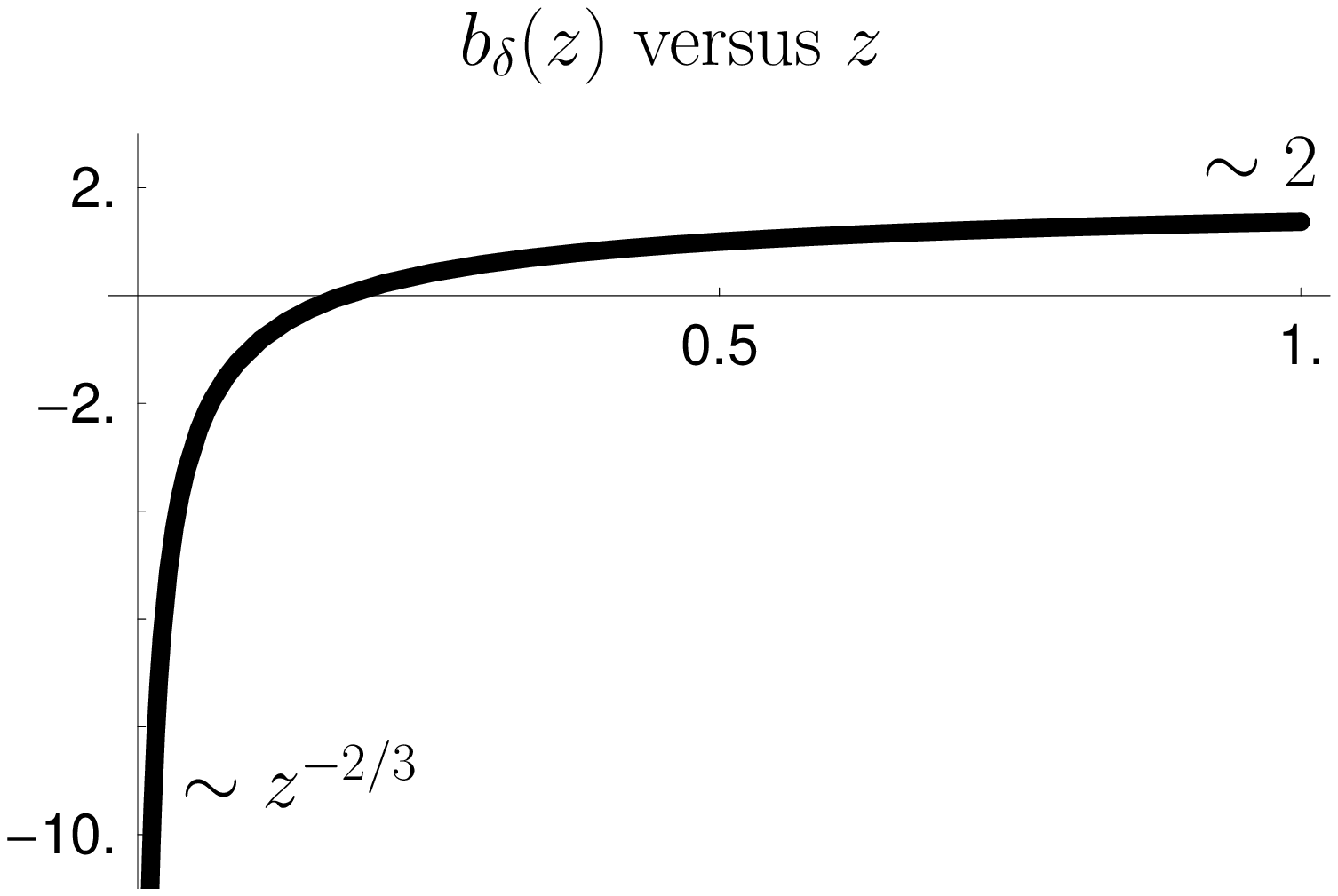}%
}%
\caption{%
Sketch of the functions $a_\delta$ and $b_\delta$ for small $\delta$
and $z\in\ccinterval{0}{1}$.
}%
\label{Fig1}%
\end{figure}%
\begin{align}
\label{eq:Def.HomSol}
\psi^{\prime}=-{a_\delta}\at{z}\,\at{b_\delta\at{z}-\teps{z}-\eps}\,\psi,\quad0\leq{z}<\infty,
\end{align}
compare \eqref{S2E10}.
For convenience we normalize $\psi$ by
\begin{align}
\label{eq:Mass.Norm}
\int\limits_{0}^{1}\,z\,\psi\at{z;\eps,\teps,\delta}\,dz=1=\int\limits_{0}^{1}\,z\,\Phi_{LSW}\at{z}\,dz=1,
\end{align}
and this implies
\begin{equation*}
\abs{\widehat{\psi}_\delta\at{z}-\Phi_{LSW}\at{z}}=\Do{1},\quad0\leq{z}\leq{1},
\end{equation*}
with $\widehat{\psi}_\delta\at{z}=\psi\at{z;0,0,\delta}$, and
\begin{equation*}
\frac{1}{\abs{\widehat{\psi}_\delta\at{z}}}\left\{
\begin{array}{lcl}
\leq{C}_\si& \text{for}&0\leq{z}\leq\tfrac{1}{2}-\si,
\\%
\geq{}c_\si\exp\at{c/\sqrt\delta}
&\text{for}&\tfrac{1}{2}+\si\leq{z}\leq{1},
\end{array}
\right.
\end{equation*}
with $0<\si<\tfrac{1}{2}$ arbitrary, compare Figure \ref{Fig2}.
\begin{figure}[ht!]%
\centering{%
\includegraphics[width=0.35\textwidth,draft=\mhdraft]{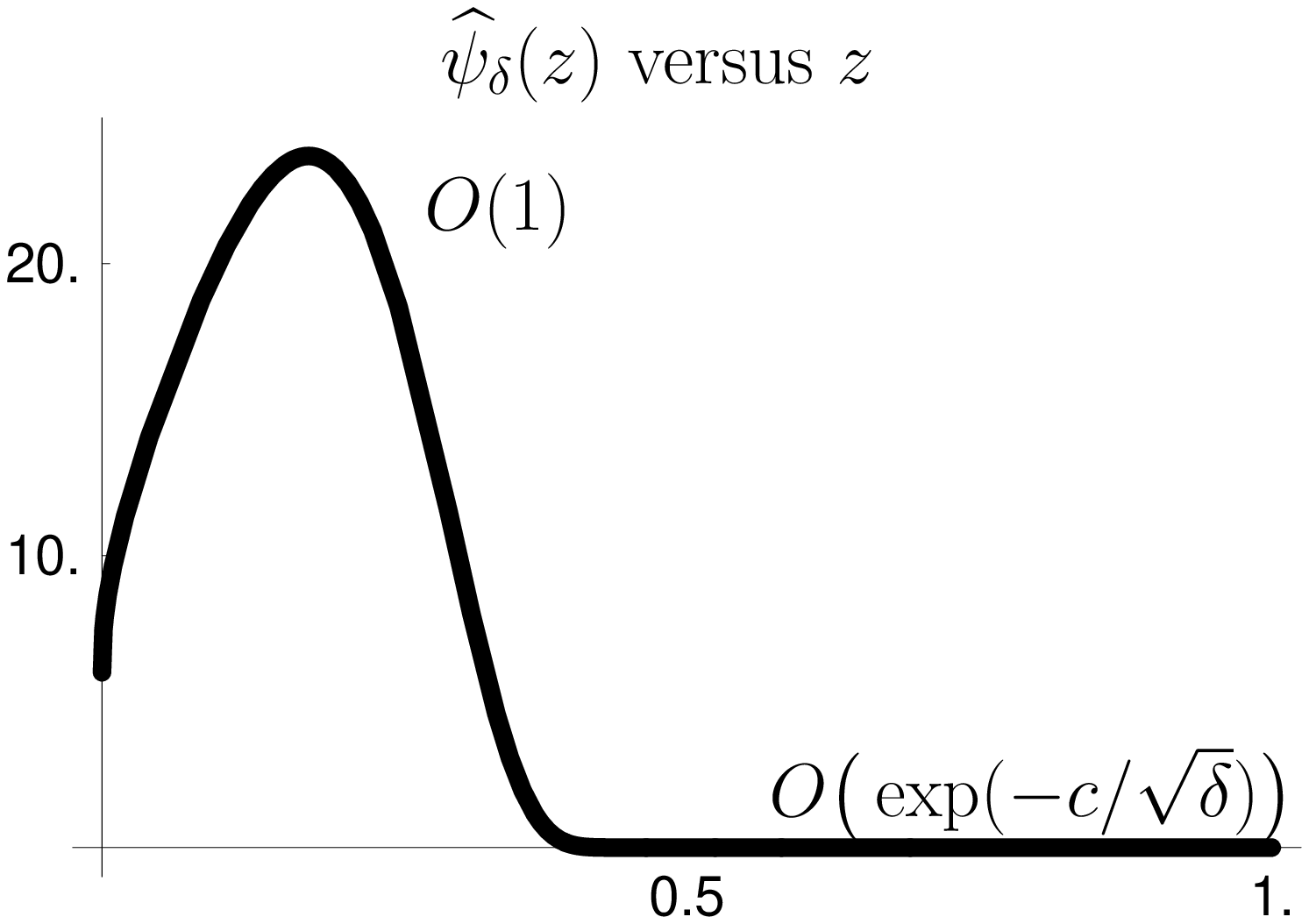}%
\hspace{0.04\textwidth}%
\includegraphics[width=0.35\textwidth,draft=\mhdraft]{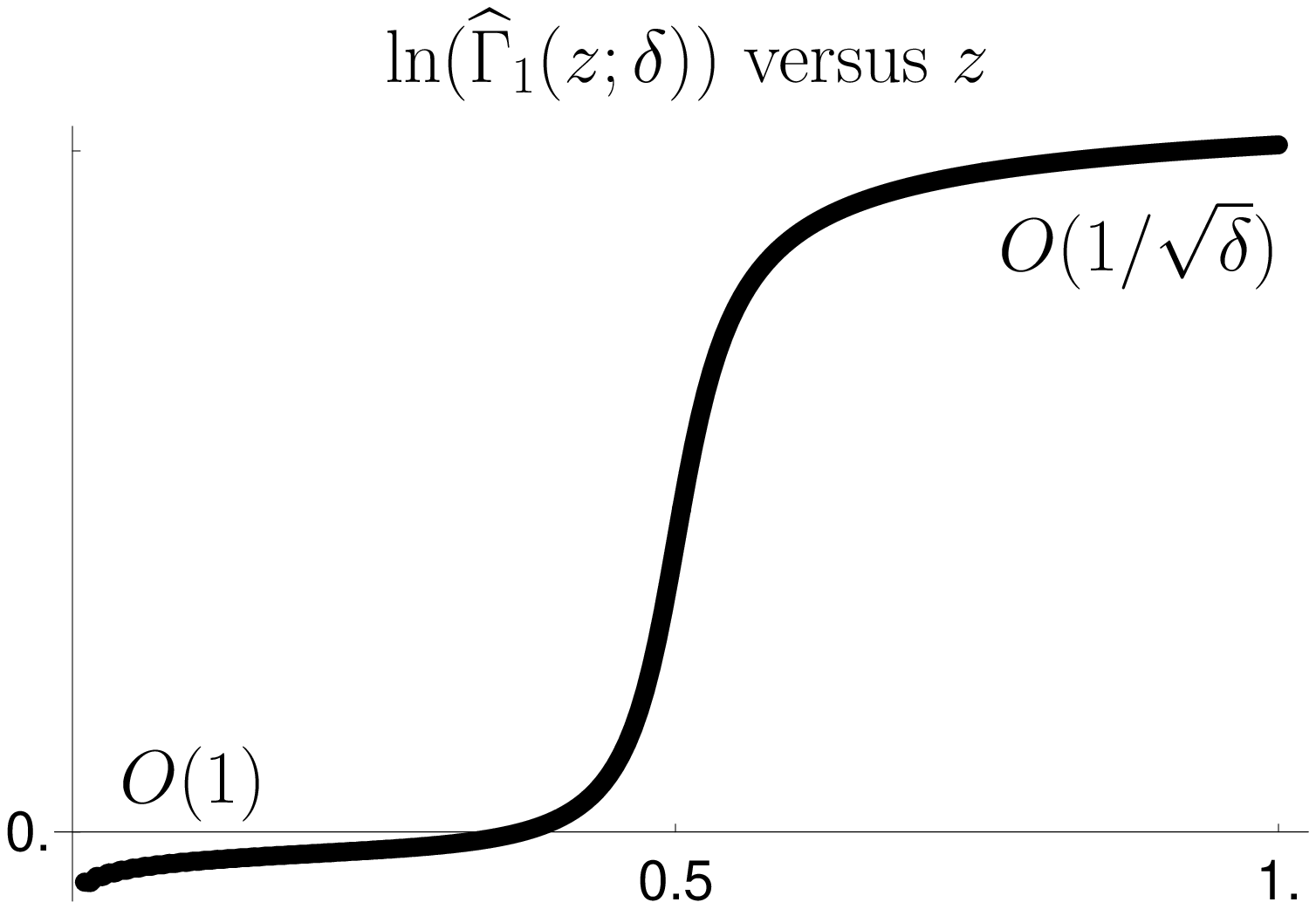}%
}%
\caption{%
Sketch of the functions and $\widehat{\psi}_\delta$ and
$\ln\nat{\widehat\Gamma_1\at{\cdot;\delta}}$ for small $\delta$ and
$z\in\ccinterval{0}{1}$.
}%
\label{Fig2}%
\end{figure}%
\bigpar%
Moreover, we define
\begin{align}
\label{Def.Ga}
\begin{split}
G_{i}[h;\eps,{\teps},\delta]&  \deq\int\limits_{0}^{\infty}%
\xi\,a_\delta\at{\xi}\,h\at\xi\,\Gamma_i\at{\xi;\eps,{\teps},\delta}d\xi,
\\%
\Gamma_i\at{\xi;\eps,{\teps},\delta}&  \deq%
\int\limits_{0}^{\xi}\ga_i\at{z}
\frac{\psi\at{z;\eps,{\teps},\delta}}{\psi\at{\xi;\eps,{\teps},\delta}}
dz
 =\int\limits_{0}^{\xi}\ga_i\at{z}
\exp\at{\int\limits_{z}^{\xi}a_\delta\at{y}\Bat{b_\delta\at{y}-{\teps}y-\eps}
\,dy}  dz,
\end{split}
\end{align}
with $\gamma_1\at{z}=z$, $\gamma_2\at{z}=1$, so that the
compatibility conditions \eqref{Def.FP.Prms.Def} read
\begin{align*}
\eps\eeq\eps^2G_1[h;\eps,\teps,\delta],\quad\quad\quad
\teps\eeq\eps^2G_2[h;\eps,{\teps},\delta].
\end{align*}
In order to prove the existence of solutions to these equations we
need careful estimates on the functionals $G_i$ and their
derivatives with respect to $\eps$ and $\teps$. These are derived
within \S\ref{sect:FP.Prms} by exploiting Assumption
\ref{Ass.Prms.1}, and rely on the following two observations. For
for small $\eps$ and $\teps$ we can ignore that the functions
$\Gamma_1$ and $\Gamma_2$ depend on these parameters, and for $h$
close to $\Phi_{LSW}*\Phi_{LSW}$ we can neglect the contributions of
the exponential tails to $G_1$ and $G_2$. More precisely, our main
approximation arguments are
\begin{align*}
\Gamma_i\at{\xi;\eps,\teps,\delta}\approx\widehat{\Gamma}_i\at{\xi;\delta},
\quad\quad\quad%
G_{i}[h;\eps,{\teps},\delta]\approx\widehat{G}_{i}[h;\delta]\approx\widehat{G}_{i}[\Phi_{LSW}*\Phi_{LSW};\delta],
\end{align*}
where
\begin{align}
\label{Def.Ga.2}
\widehat{\Gamma}_i\at{\xi;\delta}\deq\Gamma_i\at{\xi;0,0,\delta}
,\quad\quad\quad%
\widehat{G}_{i}[h;\delta]\deq\int\limits_{0}^{1}%
\xi\,a_\delta\at{\xi}\,h\at\xi\,\widehat{\Gamma}_i\at{\xi;\delta}d\xi.
\end{align}
Moreover, for small $\delta$ we can further
approximate $\widehat{G}_{i}$ by
\begin{align*}
\widehat{G}_{i}[h;\delta]\approx\cnstKi\,R_0[h],
\end{align*}
where
\begin{align*}
\cnstKi\deq\widehat{\Gamma}_i\at{1;\delta}
\end{align*}
and $R_0$ is the well defined limit for $\delta\to{0}$ of the
functionals
\begin{align}
\label{Def.Rho} %
{R}_\delta[h]\deq\int\limits_0^1\,\varrho_\delta\at{\xi}\,h\at\xi\,d\xi,\quad\quad\quad%
\varrho_\delta\at{\xi}
\deq\xi\,a_\delta\at{\xi}\exp\at{-\int\limits_\xi^1\,a_\delta\at{y}\,b_\delta\at{y}\,d{y}},
\end{align}
see Corollary \ref{Cor:Props.Rho} below.
%
%
%
%
%-----------------------------------------------------------------
\subsection{Auxiliary results}
\label{sect:AuxResults}
%-----------------------------------------------------------------
%
%-----------------------------------------------------------------
\subsubsection{Properties of the function spaces}
%-----------------------------------------------------------------
%
Here we prove that the convolution operator $*$ from \eqref{S7E4a}
is continuous and maps the space ${Z}$ into itself, and we derive
further useful estimates.
\begin{lemma}
\label{Lem:Props.H.1} %
For arbitrary $\Phi_1,\,\Phi_2\in\fspace$ we have $\Phi_1*\Phi_2\in\fspace$, and there exists
a constant $C$ such that
\begin{align*}
\fnorm{\Phi_2*\Phi_2-\Phi_1*\Phi_1}\leq
C\at{\fnorm{\Phi_2}+\fnorm{\Phi_1}}\fnorm{\Phi_2-\Phi_1}.
\end{align*}
Moreover, for each $n\in\Nset$ there exists a constant $C_n$ such that
\begin{align}
\label{Lem:Props.H.1.Eqn3} %
\frac{1}{z^n}\int\limits_{z}^{\infty}%
\xi^n\abs{\Phi\at{\xi}}\,d\xi
+\frac{1}{z^n}\int\limits_{z}^{\infty}%
\xi^n\at{\xi-z}\abs{\Phi\at{\xi}}\,d\xi
&\leq{}%
C_n\frac{\exp\at{-\beta_1{z}}}{z^{\beta_2}}\tnorm{\Phi}
\end{align}
for $z\geq{1}$ , and hence
\begin{align}
\label{Lem:Props.H.1.Eqn2} %
\int\limits_{1}^{\infty}%
\xi^n\abs{\Phi\at{\xi}}\,d\xi\leq{}C_n\tnorm{\Phi}
\end{align}
for all $\Phi\in\fspace$.
\end{lemma}
\begin{proof}
Let $\Phi_i\in{Z}$ be arbitrary. Definition \eqref{Def:Spaces}
provides
\begin{align*}
\babs{\at{\Phi_2*\Phi_1}\at{z}}
\leq{C}\fnorm{\Phi_1}\fnorm{\Phi_2},\quad\quad0\leq{z}\leq{2},
\end{align*}
where we used that
\begin{math}
\sup_{0\leq{z}\leq2}\abs{\Phi\at{z}}\leq{C}\fnorm{\Phi}.
\end{math}
For $z\geq{2}$ we estimate
\begin{align*}
\int\limits_{z-1}^{z}\frac{\exp\at{-\beta_1{y}}}{y^{\beta_2}}\,dy=
\int\limits_{0}^{1}\frac{\exp\at{-\beta_1\at{z-y}}}{\at{z-y}^{\beta_2}}\,dy
\leq%
\frac{\exp\at{-\beta_1{z}}}{\at{z-1}^{\beta_2}}
\frac{\exp\at{\beta_1}-1}{\beta_1}\leq{}%
C\frac{\exp\at{-\beta_1{z}}}{z^{\beta_2}}
\end{align*}
and
\begin{align*}
\int\limits_{1}^{z-1}\frac{\exp\at{-\beta_1\at{z-y}}}{\at{z-y}^{\beta_2}}
\frac{\exp\at{-\beta_1{y}}}{y^{\beta_2}}\,dy
\leq%
\exp\at{-\beta_1{z}}
\int\limits_{1}^{z-1}\frac{dy}{\at{z-y}^{\beta_2}\,y^{\beta_2}}
\leq{}%
C\frac{\exp\at{-\beta_1{z}}}{z^{\beta_2}}.
\end{align*}
These results imply $\Phi_2*\Phi_1\in{Z}$ with
$\fnorm{\Phi_2*\Phi_1}\leq{C}\fnorm{\Phi_1}\fnorm{\Phi_2}$, and
we conclude
\begin{align*}
\fnorm{\Phi_2*\Phi_2-\Phi_1*\Phi_1]}
&\leq%
\fnorm{\Phi_2*\at{\Phi_2-\Phi_1}}+\fnorm{\Phi_1*\at{\Phi_2-\Phi_1}}
\leq%
{C}\at{\fnorm{\Phi_2}+{\fnorm{\Phi_1}}}
\at{\fnorm{\Phi_2-\Phi_1}}.
\end{align*}
Finally, the estimates \eqref{Lem:Props.H.1.Eqn3} and \eqref{Lem:Props.H.1.Eqn2} follow from $\Phi\at{z}\leq\tnorm{\Phi}z^{-\beta_2}\exp\at{-\beta_1{z}}$ and elementary estimates for integrals.
\end{proof}
\begin{remark}
\label{Rem:Props.H.2} All constants $C$ in Lemma \ref{Lem:Props.H.1}
depend on the parameters $\beta_1$ and $\beta_2$ that appear in the
definition of the function space $\fspace$, compare
\eqref{Def:Spaces}. More precisely, we have $C\to\infty$ as
$\beta_1\to\infty$ or $\beta_2\to\infty$.
\end{remark}
%
%
%-----------------------------------------------------------------
\subsubsection{Properties of $a_\delta$ and $b_\delta$}
%-----------------------------------------------------------------
%
All subsequent estimates rely on the following properties of the
functions $a_\delta$ and $b_\delta$ which are illustrated in Figure
\ref{Fig1}.%
\begin{lemma}
\label{Lem:Prop.Prms.1}%
Let $\delta\leq{1}$ and $\si$ be arbitrary with
$0<\si<\tfrac{1}{2}$. Then,
\begin{enumerate}
\item
\label{Lem:Prop.Prms.1.Item1}%
$a_\delta$ is uniformly positive on $\ccinterval{0}{1}$, where
\begin{align*}
\max{}a_\delta\at{y}=\frac{2^{1/3}}{\delta}\at{1+\DO{\delta}}\quad\text{and}\quad
\mathrm{argmax}\,a_\delta\at{y}=\frac{1}{2}+\DO{\delta}
\end{align*}
denote the maximum and the maximizer, respectively,
\item
\label{Lem:Prop.Prms.1.Item2}%
$b_\delta$ is uniformly integrable on $\ccinterval{0}{1}$ and
nonnegative for $z\geq\at{\tfrac{1}{2}}^{5/2}$,
\item
\label{Lem:Prop.Prms.1.Item3}%
$z\,a_\delta\at{z}\to1$ and $b_\delta\at{z}\to2$ as $z\to\infty$,
\item
\label{Lem:Prop.Prms.1.Item4}%
$a_\delta$ can be
expanded with respect to $\delta$ on
$\ccinterval{0}{\tfrac{1}{2}-\si}\cup\ccinterval{\tfrac{1}{2}+\si}{1}$,
i.e.,
\begin{align*}
\abs{a_\delta\at{z}-a_0\at{z}}\leq{C_\si}\delta
\end{align*}
for $z\leq1$ with $\abs{z-\tfrac{1}{2}}>\si$ and some constant
$C_\si$ depending on $\si$,
\item
\label{Lem:Prop.Prms.1.Item5}%
\begin{math}
\displaystyle%
\int\limits_0^1a_\delta\at{y}dy=\frac{\kappa}{\sqrt\delta}\at{1+\Do{1}}
\end{math}
for some constant $\kappa$ given in the proof.
\end{enumerate}
\end{lemma}
\begin{proof}
The assertions
\ref{Lem:Prop.Prms.1.Item1}--\ref{Lem:Prop.Prms.1.Item4} follow
immediately from the definitions of $a_\delta$ and $b_\delta$,
compare \eqref{Def.Elem.Funct}. With
$y=\tfrac{1}{2}+\sqrt\delta{\eta}$ we find
\begin{align*}
\sqrt{\delta}\,a_\delta\at{y}=\frac{2}{2^{2/3}+\tfrac{4}{3}\eta^2}\at{1+\Do{1}}.
\end{align*}
This expansion implies
\begin{align*}
\sqrt\delta\int\limits_0^1{a}_\delta\at{y}\,dy=\at{1+\Do{1}}\int\limits_{-\infty}^\infty\frac{2}{2^{2/3}+\tfrac{4}{3}\eta^2}\,d\eta
=\at{1+\Do{1}}\frac{\sqrt{3}\,\pi}{2^{1/3}},
\end{align*}
and the proof is complete.
\end{proof}
%---------------------------------------------------------------------
\subsubsection{Properties of $\varrho_{\delta}$ and $R_\delta$}
%---------------------------------------------------------------------
%
%
%
\begin{lemma} %
\label{Lem:Props.Rho} %
For $\delta\leq{1}$ we have
\begin{math}
\int_0^1\varrho_\delta\at{\xi}\,d{\xi}\leq{C},
\end{math} %
where $\varrho_\delta$ is defined in \eqref{Def.Rho}. Moreover, for
each $0<\si<\tfrac{1}{2}$ there exist constants $c_\sigma$ and
$C_\sigma$ such that
\begin{enumerate}
\item $\varrho_\delta\at{\xi}\geq{c_\si}$ for
$\tfrac{1}{2}+\si\leq{\xi}\leq{1}$,
\item $\varrho_\delta\at{\xi}\leq\Do{1}\,{C_\si}$ for
$0\leq{\xi}\leq\tfrac{1}{2}-\si$.
\end{enumerate}
\end{lemma}
\begin{proof}
The existence of $c_\si$ and $C_\si$ is provided by Lemma
\ref{Lem:Prop.Prms.1}, so there remains to show
\begin{align*}
\int\limits_{1/2-\si_0}^{1/2+\si_0}\varrho_\delta\at{\xi}\,d{\xi}\leq{C}
\end{align*}
for fixed but small $\si_0$. According to Lemma
\ref{Lem:Prop.Prms.1} there exist constants $c$ and $C$ such that
\begin{align*}
\varrho_\delta\at{\xi}\leq{C}\,a_\delta\at{\xi}\exp\at{-c\int\limits_\xi^1\,a_\delta\at{y}\,d{y}}=
\frac{C}{c}\frac{\d}{\d{\xi}}\exp\at{-c\int\limits_\xi^1\,a_\delta\at{y}\,d{y}}\,d{\xi}
\end{align*}
for all $\xi$ with $\abs{\xi-\tfrac{1}{2}}\leq\si_0$,
and we conclude
\begin{align*}
\int\limits_{1/2-\si_0}^{1/2+\si_0}\varrho_\delta\at{\xi}\,d{\xi}
&\leq{}\frac{C}{c}%
\exp\at{-c\int\limits_{1/2+\si_0}^1\,a_\delta\at{y}\,d{y}}-\exp\at{-c\int\limits_{1/2-\si_0}^1\,a_\delta\at{y}\,d{y}},
\end{align*}
which gives the desired result.
\end{proof}
\begin{corollary}
\label{Cor:Props.Rho} %
The functionals $R$ are uniformly Lipschitz continuous with respect
to $h\in{Z}$ for $\delta\leq{1}$ with
\begin{align*}
R_\delta[h]\xrightarrow{\delta\to{0}}R_0[h]\deq%
\int\limits_{1/2}^1\xi\,a_0\at{\xi}\,h\at\xi\,\exp\at{-\int\limits_{\xi}^1\,a_0\at{y}\,b_0\at{y}\,d{y}}\,d\xi
\end{align*}
for all $h\in{Z}$.
\end{corollary}
%
%
%
%---------------------------------------------------------------------
\subsubsection{Properties of $\psi$}
%---------------------------------------------------------------------
%
%
%
\begin{lemma}
\label{Lem:Prop.HomSol}%
The estimates
\begin{align}
\label{Lem:Prop.HomSol.Eqn1}%
\frac{\psi\at{z;\eps_2,\teps_2,\delta}}{\psi\at{z;\eps_1,\teps_1,\delta}}\leq{}
{\exp\at{C\frac{\abs{\eps_2-\eps_1}+\abs{\teps_2-\teps_1}}{\sqrt\delta}}},
\end{align}
are satisfied for $0\leq{z}\leq{1}$, $\delta\leq{1}$, and arbitrary
$\pair{\eps_1}{\teps_1}$, $\pair{\eps_2}{\teps_2}$.
\end{lemma}
\begin{proof}
The Variation of Constants formula provides
\begin{align*}
\psi\at{z;\eps_2,\teps_2,\delta}=d\,\psi\at{z;\eps_1,\teps_1,\delta}
\exp\at{\int\limits_0^za_\delta\at{y}\at{\eps_2-\eps_1+\at{\teps_2-\teps_1}{y}}\,dy}
\end{align*}
for some factor $d$ which depends on $\eps_1$, $\eps_2$, $\teps_1$,
$\teps_2$, and $\delta$. Thanks to
\begin{align*}
\int_0^za_\delta\at{y}\,dy + \int_0^za_\delta\at{y}\,y\,dy\leq\frac{C}{2\sqrt\delta}
\end{align*}
for $0\leq{z}\leq{1}$ we conclude
\begin{align*}
d\,\widetilde{C}^{-1/2}\,
\psi\at{z;\eps_1,\teps_1,\delta}\leq{}\psi\at{z;\eps_2,\teps_2,\delta}\leq{}
d\, \widetilde{C}^{1/2} \,\psi\at{z;\eps_1,\teps_1,\delta},
\end{align*}
where $\widetilde{C}$ denotes the constant on the r.h.s. in
\eqref{Lem:Prop.HomSol.Eqn1}. Finally, the normalization condition
\eqref{eq:Mass.Norm} yields
\begin{align*}
\widetilde{C}^{-1/2}\leq{}d\leq{\widetilde{C}^{1/2}},
\end{align*}
and the proof is complete.
\end{proof}
%
%
%
%---------------------------------------------------------------------
\subsubsection{Properties of $\widehat{\Gamma}_{i}$ and $\Gamma_{i}$}
%---------------------------------------------------------------------
%
%
Recall Definition \eqref{Def.Ga.2}, which implies that the functions
$\xi\mapsto\widehat{\Gamma}_i\at{\xi;\delta}$, $i=1,2$, are strictly
increasing and satisfy the ODE
\begin{align}
\label{Lem:Prop.GaH.Eqn1}
\VarDer{\widehat{\Gamma}_i\at{\xi;\delta}}{\xi}=
\gamma_i\at\xi+a_\delta\at{\xi}b_\delta\at{\xi}\widehat{\Gamma}_i\at{\xi;\delta}>0,\quad
\widehat{\Gamma}_i\at{0;\delta}=0.
\end{align}
\begin{lemma}
For all $\delta\leq1$ we have
\begin{align}
\label{Lem:Prop.GaH.2.Eqn1} %
{c}\,\exp\at{\frac{c}{\sqrt{\delta}}}\leq
\cnstKi\leq{C}\,\exp\at{\frac{C}{\sqrt{\delta}}},
\end{align}
and
\begin{align}
\label{Lem:Prop.GaH.2.Eqn3} %
c\,\cnstKT
\leq{}%
\cnstKO
\leq%
\cnstKT.
\end{align}
\end{lemma}
\begin{proof}
Exploiting  the properties of $a_\delta$ and $b_\delta$ from Lemma
\ref{Lem:Prop.Prms.1} we find
\begin{align*}
\widehat{\Gamma}_{i}\at{1;\delta}%
\leq\int\limits_0^{1}\ga_i\at{z}%
\exp\at{\int\limits_0^{\tfrac{1}{4}}a_\delta\at{y}b_\delta\at{y}\,dy}
\exp\at{\int\limits_{\tfrac{1}{4}}^{1}a_\delta\at{y}b_\delta\at{y}\,dy}
\,dz
\leq%
C\exp\at{\frac{C}{\sqrt{\delta}}},
\end{align*}
as well as
\begin{align*}
\widehat{\Gamma}_{i}\at{1;\delta}%
\geq%
c\int\limits_0^{1/4}\ga_i\at{z}\exp\at{c\int\limits_{1/4}^{1}a_\delta\at{y}\,dy}\,dz
\geq{c}\exp\at{\frac{c}{\sqrt{\delta}}},
\end{align*}
and this gives \eqref{Lem:Prop.GaH.2.Eqn1} since $\cnstKi=\widehat{\Gamma}_{i}\at{1;\delta}$ by definition. The inequality $\cnstKO\leq\cnstKT$ is
obvious as $\ga_1\at{z}\leq\ga_2\at{z}$ for all $0\leq{z}\leq{1}$.
Moreover, there exists a constant $\tilde{c}$ such that
\begin{align*}
\widehat{\Gamma}_1\at{\tfrac{1}{4};\delta}\geq
\tilde{c}\,\widehat{\Gamma}_2\at{\tfrac{1}{4};\delta},
\end{align*}
where we used that $a_\delta$, $b_\delta$ can be expanded in powers
of $\delta$ on $\cinterval{0}{\tfrac{1}{4}}$, and the ODE
\eqref{Lem:Prop.GaH.Eqn1} implies that $\widehat{\Gamma}_1$ is a
supersolution to the equation for
$c\,\widehat{\Gamma}_2\at{\xi;\delta}$ on
$\ccinterval{\tfrac{1}{4}}{1}$, where
$c=\min\{\tilde{c},\,\tfrac{1}{4}\}$. Hence we proved \eqref{Lem:Prop.GaH.2.Eqn3}.
\end{proof}
\begin{lemma}
\label{Lem:Prop.GaH.2}
Let $\delta\leq1$. Then
\begin{align}
\label{Lem:Prop.GaH.2.Eqn4}%
\widehat{\Gamma}_{i}\at{\xi;\delta}\leq%
C\,\cnstKi\xi^4,\quad\quad1\leq{\xi}<\infty,
\end{align}
and
\begin{align}
\label{Lem:Prop.GaH.2.Eqn2} %
\widehat{\Gamma}_i\at{\xi;\delta} &=\cnstKi\at{
\exp\at{-\int\limits_\xi^1\,a_\delta\at{y}\,b_\delta\at{y}\,dy}-\Do{1}\sqrt\delta},\quad\quad0\leq\xi\leq{1}.%
\end{align}%
In particular,
\begin{enumerate}
\item
$\widehat{\Gamma}_i\at{\xi;\delta}\leq{C_\si}$ for
$\xi\leq{}\tfrac{1}{2}-\si$,
\item
$\widehat{\Gamma}_i\at{\xi;\delta}\geq{}c_\si\cnstKi$ for
$\tfrac{1}{2}+\si\leq\xi\leq{1}$,
\end{enumerate}
with $0<\si<\tfrac{1}{2}$ arbitrary.
\end{lemma}
\begin{proof} We start with $\xi\geq1$. Lemma \ref{Lem:Prop.Prms.1}
provides%
\begin{align*}
\exp\at{\int\limits_1^\xi\,a_\delta\at{y}\,b_\delta\at{y}\,dy}\leq\exp\at{C+3\ln\at{\xi}}\leq{C}\xi^3
\end{align*}
and we conclude
\begin{align*}
\widehat{\Gamma}_i\at{\xi;\delta}
&=%
\int\limits_0^1 \ga_i\at{z}
\exp\at{\int\limits_z^\xi\,a_\delta\at{y}\,b_\delta\at{y}\,dy}%
\,dz + \int\limits_1^\xi \ga_i\at{z}
\exp\at{\int\limits_z^\xi\,a_\delta\at{y}\,b_\delta\at{y}\,dy}%
\,dz
\\&\leq%
C\xi^3 \int\limits_0^1 \ga_i\at{z}
\exp\at{\int\limits_z^1\,a_\delta\at{y}\,b_\delta\at{y}\,dy} \,dz +
C\xi^3  \int\limits_1^\xi \ga_i\at{z}\,dz
\leq%
\,C\xi^4\at{\cnstKi+1},
\end{align*}
which implies \eqref{Lem:Prop.GaH.2.Eqn4} thanks to
\eqref{Lem:Prop.GaH.2.Eqn1}.
Now consider $0\leq{}\xi\leq1$, and let
$\xi_0=\at{\tfrac{1}{2}}^{5/2}<\tfrac{1}{2}$ so that
$b_\delta\at{y}\geq0$ for all $y\geq{\xi_0}$, compare Lemma
\ref{Lem:Prop.Prms.1}. For $0\leq\xi\leq\xi_0$ we have
\begin{align*}
0&\leq%
\int\limits_\xi^1\ga_i\at{z}
\exp\at{\int\limits_z^\xi\,a_\delta\at{y}\,b_\delta\at{y}\,dy}%
\\&=%
\int\limits_\xi^{\xi_0}\ga_i\at{z}
\exp\at{-\int\limits_\xi^z\,a_\delta\at{y}\,b_\delta\at{y}\,dy}%
+\int\limits_{\xi_0}^1\ga_i\at{z}
\exp\at{-\int\limits_\xi^z\,a_\delta\at{y}\,b_\delta\at{y}\,dy}%
\,dz
\\&\leq%
\int\limits_\xi^{\xi_0}\ga_i\at{z}
\exp\at{\int\limits_0^{\xi_0}\,a_\delta\at{y}\,\abs{b_\delta\at{y}}\,dy}
\,dz + \int\limits_{\xi_0}^1\ga_i\at{z} \,dz\leq{C},
\end{align*}%
whereas for $\xi_0\leq\xi\leq1$ we find
\begin{align*}
0\leq%
\int\limits_\xi^1\ga_i\at{z}
\exp\at{\int\limits_z^\xi\,a_\delta\at{y}\,b_\delta\at{y}\,dy}%
\,dz
\leq%
\int\limits_\xi^1\ga_i\at{z} \,dz \leq{C}.
\end{align*}%
Therefore, for all $0\leq\xi\leq{1}$ we have
\begin{align*}
\widehat{\Gamma}_i\at{\xi;\delta} &=\int\limits_0^1 \ga_i\at{z}
\exp\at{\int\limits_z^\xi\,a_\delta\at{y}\,b_\delta\at{y}\,dy}%
\,dz- \int\limits_\xi^1 \ga_i\at{z}
\exp\at{\int\limits_z^\xi\,a_\delta\at{y}\,b_\delta\at{y}\,dy}%
\,dz
\\&\geq%
\int\limits_0^1 \ga_i\at{z}
\exp\at{\int\limits_z^1\,a_\delta\at{y}\,b_\delta\at{y}\,dy}%
\exp\at{-\int\limits_\xi^{1}\,a_\delta\at{y}\,b_\delta\at{y}\,dy}%
\,dz-C
\\&= %
\cnstKi
\exp\at{-\int\limits_\xi^1\,a_\delta\at{y}\,b_\delta\at{y}\,dy}%
-C,
\end{align*}%
and \eqref{Lem:Prop.GaH.2.Eqn2} follows due to
\eqref{Lem:Prop.GaH.2.Eqn1}. The remaining assertions are direct
consequences of \eqref{Lem:Prop.GaH.2.Eqn2} and Lemma
\ref{Lem:Prop.Prms.1}.
\end{proof}
From now on we assume that both $\eps$ and $\teps$ are small with
respect to $\delta$.
\begin{assumption}
\label{Ass:Prms.0} Suppose $\eps=\nDo{\sqrt\delta}$ and
$\teps=\nDo{\sqrt\delta}$.
\end{assumption}
\begin{lemma}
\label{Lem:Props.Ga} %
Under Assumption \ref{Ass:Prms.0} the estimates
\begin{align*}
\Gamma_i\at{\xi;\eps,{\teps},\delta}\leq{}\widehat{\Gamma}_i\at{\xi;\delta},
\quad0\leq\xi<\infty,%
\end{align*}
and
\begin{align}
\notag%\label{Lem:Props.Ga.Eqn1} %
\abs{\Gamma_i\at{\xi;\eps,{\teps},\delta}-\widehat{\Gamma}_i\at{\xi;\delta}}%
=\Do{1}\widehat{\Gamma}_i\at{\xi;\delta},\quad 0\leq\xi\leq{1},
\end{align}
hold for $i=1,2$ and all $\delta\leq1$.%
\end{lemma}
\begin{proof}
The first assertion follows from Definition \eqref{Def.Ga} and the
positivity of $a_\delta$. With $0\leq{z}\leq{\xi}\leq{1}$ we find
\begin{align*}
0\leq\int\limits_z^\xi{a}_\delta\at{y}\at{{\teps}y+\eps}\,dy
\leq{}\at{{\teps}+\eps}\int\limits_0^1{a}_\delta\at{y}
\leq{}C\frac{\eps+{\teps}}{\sqrt\delta},%
\end{align*}
and %
\begin{align*}
\abs{\exp\at{-\int\limits_z^\xi{a}_\delta\at{y}\at{{\teps}y+\eps}\,dy}-1}
&\leq%
\exp\at{\int\limits_z^\xi{a}_\delta\at{y}\at{{\teps}y+\eps}\,dy}-1\leq\exp\at{C\frac{\eps+\teps}{\sqrt\delta}}-1
=\Do{1}
\end{align*}
gives the second assertion.
\end{proof}
%
%
%
%
%-----------------------------------------------------------------
\subsection{Solving the fixed point equation for $\pair{\eps}{\teps}$}
\label{sect:FP.Prms}
%-----------------------------------------------------------------%
%
\begin{lemma}
\label{Lem:Restr.Prms} %
For $\delta\leq1$ and all $h$ sufficiently close to
$\Phi_{LSW}*\Phi_{LSW}$ the following estimates are satisfied
\begin{align}
\label{Lem:Restr.Prms.Eqns} %
\fnorm{h}\leq2\fnorm{\Phi_{LSW}*\Phi_{LSW}},\quad\quad
\widehat{G}_i[h;\delta]=\bat{1\pm\Do{1}}\,{\cnstKi}R_\delta[h],\quad\quad
c
\leq%
{R}_\delta[h]\leq{}C.
\end{align}
\end{lemma}
\begin{proof}
Equation \eqref{Lem:Prop.GaH.2.Eqn2} provides
\begin{align}
\label{Lem:Props.G.Func.2.Eqn3} %
\widehat{G}_i[h;\delta]&=\int\limits_0^1\,\xi\,a_\delta\at\xi\,h\at\xi\,\widehat{\Gamma}_i\at{\xi;\delta}\,d\xi
\notag\\&=%
\cnstKi\,\int\limits_0^1\,\xi\,a_\delta\at\xi\,h\at\xi\,
\exp\at{-\int\limits_\xi^1\,a_\delta\at{y}\,b_\delta\at{y}\,dy}
\,d\xi
- \Do{1}{\sqrt{\delta}}\,\cnstKi
\int\limits_0^1\,\xi\,a_\delta\at\xi\,h\at\xi\, \,d\xi,
\notag\\&=%
\cnstKi\at{R_\delta[h]\pm\Do{1}\snorm{h}}.
\end{align}
Recall that $\Phi_{LSW}*\Phi_{LSW}$ is strictly positive on
$\ccinterval{0}{\tfrac{7}{8}}$, and suppose that
$\fnorm{h-\Phi_{LSW}*\Phi_{LSW}}$ is sufficiently small such that
$h\at{\xi}\leq2\fnorm{\Phi_{LSW}*\Phi_{LSW}}$ for all
$0\leq\xi\leq1$ and
\begin{align*}
h\at{\xi}\geq\tfrac{1}{2}\at{\Phi_{LSW}*\Phi_{LSW}}\at{\xi}
\geq{c},\qquad\tfrac{5}{8}\leq\xi\leq\tfrac{7}{8}.
\end{align*}
Thanks to Lemma \ref{Lem:Props.Rho} this estimate implies
\begin{align*}
0<c\int_{5/8}^{7/8}\varrho_\delta\at\xi\,d\xi\leq\,R_\delta[h]\leq
2\fnorm{\Phi_{LSW}*\Phi_{LSW}}R_\delta[1]<\infty.
\end{align*} %
In particular,
$R_{\delta}[h]\pm\Do{1}\snorm{h}=\at{1\pm\Do{1}}R_{\delta}[h]$, which is the third claimed estimate,
and using \eqref{Lem:Props.G.Func.2.Eqn3} we find the second inequality from \eqref{Lem:Restr.Prms.Eqns}.
\end{proof}
From now on we make the following assumption on the function $h$.
\begin{assumption}
\label{Ass:Prms} Let $\delta$ be
sufficiently small, and $\eps+\teps=\nDo{\sqrt\delta}$. Moreover, suppose that
$h$ is
sufficiently close to $\Phi_{LSW}*\Phi_{LSW}$ in the sense that all estimates from
\eqref{Lem:Restr.Prms.Eqns} as well as
\begin{align}
\label{Lem:Restr.Prms.Eqn3}%
\tnorm{h}=\Do{1}
\end{align}
are satisfied.
\end{assumption}
\begin{remark}
\label{Ass:Prms.Rem}
\begin{enumerate}
\item The condition $\tnorm{h}=\Do{1}$ arises
naturally. In fact, below we consider functions $\Phi$ with $\fnorm{\Phi-\Phi_{LSW}}=\Do{1}$, and this implies
$\tnorm{\Phi*\Phi}\leq\fnorm{\Phi*\Phi-\Phi_{LSW}*\Phi_{LSW}}=\Do{1}$.
\item
The condition $\fnorm{\Phi-\Phi_{LSW}}=\Do{1}$ depends crucially on
the parameters parameter $\beta_1$ and $\beta_2$ from
\eqref{Def:Spaces}, because for fixed $\Phi$ the quantity
$\tnorm{\Phi}$ grows as $\beta_1\to\infty$ or $\beta_2\to\infty$.
This will effect the final choice for $\delta$, see Condition
\eqref{Def:PrmMu} below.
\end{enumerate}
\end{remark}%---------------------------------------------------------------------
\subsubsection{Properties of $G_{i}$,  $\widehat{G}_{i}$, and their derivatives}
%---------------------------------------------------------------------
%
%
%
%
%
%
In order to estimate $G_i$ we split the $\xi$-integration in \eqref{Def.Ga} as follows
\begin{align}
\notag%\label{Lemma:Props.G.Func.Eqn0} %
\begin{split}
G_{i,1}[h;\eps,{\teps},\delta]\deq\int\limits_{0}^{1}%
\xi\,a_\delta\at{\xi}\,h\at\xi\,{\Gamma}_i\at{\xi;\eps,{\teps},\delta}d\xi,
\quad%
G_{i,2}[h;\eps,{\teps},\delta]\deq\int\limits_{1}^{\infty}%
\xi\,a_\delta\at{\xi}\,h\at\xi\,{\Gamma}_i\at{\xi;\eps,{\teps},\delta}d\xi,
\end{split}
\end{align}
so that $G_i=G_{i,1}+G_{i,2}$.
\begin{lemma}
\label{Lemma:Props.G.Func} %
Assumption \ref{Ass:Prms} implies
\begin{align*}
\abs{G_i[h;\eps,{\teps},\delta]-\widehat{G}_i[h;\delta]}\leq\Do{1}\,\widehat{G}_i[h;\delta],
\end{align*}
and
\begin{align*}
\eps\,\babs{\partial_\eps{G_{i}[h;\eps,{\teps},\delta]}}
\leq%
{\Do{1}}\,\widehat{G}_i[h;\delta],\quad\quad
{\teps}\,\babs{\partial_{{\teps}}{G_{i}[h;\eps,{\teps},\delta]}}
\leq%
\Do{1}\,\widehat{G}_i[h;\delta],
\end{align*}
for both  $i=1$ and $i=2$.
\end{lemma}
\begin{proof}
Lemma \ref{Lem:Props.Ga} implies %
\begin{align}
\label{Lemma:Props.G.Func.Eqn1} %
\abs{G_{i,1}[h;\eps,{\teps},\delta]-\widehat{G}_{i}[h;\delta]}\leq\Do{1}\,\widehat{G}_{i}[h;\delta],
\end{align}
and using Lemma \ref{Lem:Prop.Prms.1} and Lemma \ref{Lem:Prop.GaH.2}
we find
\begin{align}
\label{Lemma:Props.G.Func.Eqn2} %
\begin{split}
G_{i,2}[h;\eps,{\teps},\delta]
&\leq{}%
\int\limits_{1}^{\infty}%
\xi\,a_\delta\at{\xi}\,h\at\xi\,\widehat{\Gamma}_i\at{\xi;\delta}\,d\xi
\leq%
C\,\cnstKi%
\int\limits_{1}^{\infty}%
h\at\xi\,\xi^4\,d\xi
\\&\leq%
C\,\cnstKi\tnorm{h}%
\leq%
\Do{1}\widehat{G}_i\at{h;\delta},%
\end{split}
\end{align}
where we additionally used \eqref{Lem:Props.H.1.Eqn2} and \eqref{Lem:Restr.Prms.Eqn3}.
Combining \eqref{Lemma:Props.G.Func.Eqn1} and \eqref{Lemma:Props.G.Func.Eqn2}
yields the desired estimates for $G_i$. In order to control the
derivatives we compute
\begin{align*}
\VarDer{G_{i,1}}{\eps}[h;\eps,{\teps},\delta]%
&=\int\limits_0^{1}\xi\,a_\delta\at\xi\,h\at\xi\,
\VarDer{\Gamma_i}{\eps}\at{\xi;\eps,{\teps},\delta}\,d\xi,
\\%
\VarDer{G_{i,1}}{\teps}[h;\eps,{\teps},\delta]%
&=\int\limits_0^{1}\xi\,a_\delta\at\xi\,h\at\xi\,
\VarDer{\Gamma_i}{\teps}\at{\xi;\eps,{\teps},\delta}\,d\xi,
\end{align*}
as well as similar formulas for the derivatives of $G_{i,2}$, where
\begin{align*}
\VarDer{\Gamma_i}{\eps}\at{\xi;\eps,{\teps},\delta}
&=\int\limits_0^\xi\ga_i\at{z}%
\exp\at{\int\limits_z^\xi{a}_\delta\at{y}\at{b_\delta\at{y}-{\teps}y-\eps}\,dy}%
\at{-\int\limits_z^\xi{a}_\delta\at{y}\,dy}%
\,dz, \\%
\VarDer{\Gamma_i}{\teps}\at{\xi;\eps,{\teps},\delta}
&=\int\limits_0^\xi\ga_i\at{z}%
\exp\at{\int\limits_z^\xi{a}_\delta\at{y}\at{b_\delta\at{y}-{\teps}y-\eps}\,dy}%
\at{-\int\limits_z^\xi{a}_\delta\at{y}\,y\,dy}%
\,dz.
\end{align*}
For $0\leq{z}\leq\xi\leq{1}$ we estimate
\begin{align*}
\abs{{\int\limits_z^\xi{a}_\delta\at{y}\at{y+1}\,dy}}%
&\leq\int\limits_0^{1}{a}_\delta\at{y}\at{y+1}\,dy%
\leq\frac{C}{\sqrt\delta},\quad
\end{align*}
so that
\begin{align*}
\babs{\VarDer{\Gamma_i}{\eps}\at{\xi;\eps,{\teps},\delta}}
\leq%
\frac{C}{\sqrt\delta}\,\widehat{\Gamma}_i\at{\xi;\delta} ,\quad
\babs{\VarDer{\Gamma_i}{\teps}\at{\xi;\eps,{\teps},\delta}}
\leq%
\frac{C}{\sqrt\delta}\,\widehat{\Gamma}_i\at{\xi;\delta}
\end{align*}
due to Lemma \ref{Lem:Props.Ga}. Multiplication with $\xi\,a_\delta\at\xi\,h\at\xi$ and integration over $0\leq\xi\leq{1}$ provides
\begin{align*}
\eps\babs{\VarDer{G_{i,1}}{\eps}[h;\eps,{\teps},\delta]}\leq%
\Do{1}\,\widehat{G}_i[h;\delta] ,\quad
{\teps}\babs{\VarDer{G_{i,1}}{\eps}[h;\eps,{\teps},\delta]}
\leq%
\Do{1}\,\widehat{G}_i[h;\delta],
\end{align*}
where we used $\at{\eps+\teps}/\sqrt\delta=\Do{1}$.
For
$\xi\geq{1}$ we estimate
\begin{align*}
\abs{{\int\limits_z^\xi{a}_\delta\at{y}\at{y+1}\,dy}}%
&\leq\int\limits_0^{1}{a}_\delta\at{y}\at{y+1}\,dy+
\int\limits_1^{\xi}{a}_\delta\at{y}\at{y+1}\,dy
\leq\frac{C}{\sqrt\delta}+C\xi\leq\frac{C}{\sqrt\delta}\xi,
\end{align*}
and exploiting Lemma \ref{Lem:Prop.GaH.2} and Lemma \ref{Lem:Props.Ga} we find
\begin{align*}
\abs{\partial_\eps{\Gamma}_i\at{\xi;\eps,{\teps},\delta}}
\leq%
\frac{C}{\sqrt\delta}\,\xi^5\,\cnstKi ,\quad
\abs{\partial_{{\teps}}{\Gamma}_i\at{\xi;\eps,{\teps},\delta}}
\leq%
\frac{C}{\sqrt\delta}\,\xi^5\,\cnstKi.
\end{align*}
Finally, \eqref{Lem:Props.H.1.Eqn2} combined with \eqref{Lem:Restr.Prms.Eqns} gives
\begin{align*}
\eps\abs{\partial_{\eps}{G_{i,2}[h;\eps,{\teps},\delta]}}
\leq%
\Do{1}\,\widehat{G}_i[h;\delta] ,\quad
{\teps}\abs{\partial_{{\teps}}{G_{i,2}[h;\eps,{\teps},\delta]}}
\leq%
\Do{1}\,\widehat{G}_i[h;\delta],
\end{align*}
and the proof is complete.
\end{proof}%
%
%
%
%
%---------------------------------------------------------------------
\subsubsection{Choosing $\eps$ and $\teps$}
%---------------------------------------------------------------------%
\begin{lemma}
\label{Lem:FP.Prms.Sol}
The following assertions hold true under
Assumption \ref{Ass:Prms}.
\begin{enumerate}
\item
For small $\delta$ and $i=1,2$ we have
\begin{align}
\label{Lem:FP.Prms.Sol.Eqn1a}%
\babs{g_i[h;\eps,\teps,\delta]-\eps^2\widehat{G}_i[h;\delta]}
\leq\Do{1}\,\eps^2\widehat{G}_i[h;\delta]
\end{align}
as well as
\begin{align}
\label{Lem:FP.Prms.Sol.Eqn1b}%
\babs{\eps\VarDer{g_i}{\eps}[h;\eps,\teps,\delta]-2g_i[h;\eps,\teps,\delta]}
\;\;\;+\;\;\;%
\babs{{\teps}\,\VarDer{g_i}{\teps}[h;\eps,\teps,\delta]}
\;\leq\;\Do{1}\,\eps^2\widehat{G}_i[h;\delta].
\end{align}
\item
For each $\alpha>1$ there exists $\delta_\alpha>0$ such that for all
$\delta\leq\delta_\alpha$ each solution $\pair{\eps}{{\teps}}$ to
the compatibility conditions \eqref{Def.FP.Prms} must belong to
\begin{align*}
U[h;\alpha,\delta]\deq\left\{%
\pair{\eps}{\tilde{\eps}}\;:\;\tfrac{1}{\alpha}\,\eps_\app[h;\delta]\leq\eps\leq\alpha\,\eps_\app[h;\delta],\;%
\tfrac{1}{\alpha}\,\teps_\app[h;\delta]\leq{\teps}\leq\alpha\,\teps_\app[h;\delta]
\right\},%
\end{align*}
where
\begin{align*}
\eps_\app[h;\delta]\deq1/\widehat{G}_1[h;\delta],\quad\quad
{\teps}_\app[h;\delta]\deq\eps_\app[h;\delta]\,\widehat{G}_2[h;\delta]/\widehat{G}_1[h;\delta]
\end{align*}
have the same order of magnitude thanks to
\eqref{Lem:Prop.GaH.2.Eqn3}.
\item
For small $\delta$ and given $h$ (bounded by Assumption
\ref{Ass:Prms}) there exists a solution
$\pair{\eps[h;\delta]}{\teps[h;\delta]}$ to \eqref{Def.FP.Prms}.
Moreover, this solution is unique under the constraints
$\eps,\,\teps=\nDo{\sqrt\delta}$.
\end{enumerate}
\end{lemma}
\begin{proof}
Let $\delta$ be sufficiently small, and $h$ and
$\alpha>1$ be fixed. The estimates \eqref{Lem:FP.Prms.Sol.Eqn1a} and \eqref{Lem:FP.Prms.Sol.Eqn1b} are
provided by Lemma \ref{Lemma:Props.G.Func} and imply
\begin{align*}
\frac{g_1\pair{\eps}{\teps}}{\eps}=\at{1\pm\Do{1}}\frac{\eps}{\eps_\app},\quad
g_2\pair{\eps}{\teps}=\at{1\pm\Do{1}}\at{\frac{\eps}{\eps_\app}}^2\teps_\app.
\end{align*}
where $g_i\pair{\eps}{\teps}$, $\eps_\app$, and $\teps_\app$ are
shorthand for $g_i[h,\eps,\teps,\delta]$, $\eps_\app[h;\delta]$, and
$\teps_\app[h;\delta]$, respectively.  Therefore,
$\pair{g_1\pair{\eps}{\teps}}{g_2\pair{\eps}{\teps}}=\pair{\eps}{\teps}$
implies
$\eps=\at{1\pm\Do{1}}\eps_\app$, and in turn
$\teps=\at{1\pm\Do{1}}\teps_\app$, so each solution to
\eqref{Def.FP.Prms} must be an element of $U[h;\alpha,\delta]$.
\par
Now suppose $\pair{\eps}{\teps}\in{U}[h,\al,\delta]$. For
$\eps=\tfrac{1}{\alpha}\eps_\app$ and $\eps={\alpha}\eps_\app$ we have
\begin{align*}
g_1\pair{\eps}{\teps}=\at{1\pm\Do{1}}\tfrac{1}{\al}\eps<\eps\qquad\text{and}\qquad
g_1\pair{\eps}{\teps}=\at{1\pm\Do{1}}{\al}\eps>\eps,
\end{align*}
respectively, and \eqref{Lem:FP.Prms.Sol.Eqn1b} implies
$\VarDer{g_1}{\eps}>0$.
Therefore, for each $\teps$ there exists a unique solution
$\eps=\eps\at{\teps}$ to $g_1\eeq\eps$, i.e.,
\begin{align*}
g_1\pair{\eps\at\teps}{\teps}=\eps\at{\teps}=\at{1\pm\Do{1}}\eps_\app,
\end{align*} %
and differentiation with respect to $\teps$ shows
\begin{align*}
\abs{\frac{\d\eps}{\d\teps}}=\abs{\VarDer{g_1}{\teps}}/\abs{\VarDer{g_1}{\eps}-1}=
\Do{1}\frac{\eps}{\teps}=\Do{1}\frac{\eps_\app}{\teps_\app}=\Do{1}
\end{align*} %
since $c\eps_\app\leq\teps_\app\leq{C}\eps_\app$ thanks to Lemma
\ref{Lem:Restr.Prms}. Now let
$\tilde{g}_2\at{\teps}\deq{}g_2\pair{\eps\at\teps}{\teps}$, and note
that
\begin{align*}
\tilde{g}_2=\at{1\pm\Do{1}}\teps_\app
,\quad%
\abs{\frac{\tilde{g}_2}{\d\teps}} =
\abs{\VarDer{g_2}{\eps}}\abs{\frac{\d\eps}{\d\teps}}+\abs{\VarDer{g_2}{\teps}}=\Do{1}\frac{\tilde{g_2}}{\tilde\eps}=\Do{1}.
\end{align*}
Thus, for small $\delta$ the function $\tilde{g_2}$ is contractive
with
\begin{align*}
\tilde{g}_2\at{\tfrac{1}{\al}\teps_\app}>\tfrac{1}{\al}\teps_\app,\quad%
\tilde{g}_2\at{{\al}\teps_\app}<{\al}\teps_\app,
\end{align*}
and hence there exists a unique solution to
$\teps\eeq\tilde{g}_2\at\teps$. %
\end{proof}
In Section \ref{sect:FP.Phi} below we consider functions
$h$ close to $\Phi_{LSW}*\Phi_{LSW}$, and then the following
result, which follows from
Lemma \ref{Lem:Restr.Prms},
Lemma \ref{Lem:FP.Prms.Sol} and Corollary
\ref{Cor:Props.Rho}, becomes useful.
\begin{corollary}
\label{Cor:FP.Prms.Sol.2}
 Suppose that
\begin{align*}
\fnorm{h-\Phi_{LSW}*\Phi_{LSW}}=\Do{1}.
\end{align*}
Then the
solution $\pair{\eps[h;\delta]}{\teps[h;\delta]}$ from Lemma
\ref{Lem:FP.Prms.Sol}  satisfies
\begin{align*}
\eps_\app[h;\delta]=\at{1\pm\Do{1}}\epsilon_\delta,\quad\quad
\teps_\app[h;\delta]=\at{1\pm\Do{1}}\widetilde{\epsilon}_\delta,
\end{align*}
where
\begin{align}
\label{Eqn:Eps.Asymp}
\epsilon_\delta\deq1/\,{\cnstKO}R_0[\Phi_{LSW}*\Phi_{LSW}],\quad\quad
\widetilde{\epsilon}_\delta\deq\epsilon_\delta\,{\cnstKT}/\,\cnstKO.
\end{align}
In particular, the assertions from Theorem \ref{MainTheo1} and Proposition \ref{Intro.Prop1} are satisfied.
\end{corollary}
%
%-----------------------------------------------------------------
\subsubsection{Continuity of $\eps$ and $\tilde{\eps}$.}
%-----------------------------------------------------------------
%
%
%
\begin{lemma}
\label{Lem:FP.Prms.Cont}
The solution from Lemma \ref{Lem:FP.Prms.Sol} depends
Lipschitz-continuously on $h$. More precisely, for arbitrary $h_1$,
$h_2$ that fulfil Assumption \ref{Ass:Prms} we have
\begin{align*}
\Babs{\eps[h_2;\delta]-\eps[h_1;\delta]} +
\Babs{{\teps}[h_2;\delta]-\teps[h_1;\delta]}
&\leq%
C\at{\eps[h_2;\delta]+\eps[h_1;\delta]}\,\fnorm{h_2-h_1}.
\end{align*}
\end{lemma}
\begin{proof}
We fix $\delta$, abbreviate $\eps_i=\eps[h_i;\delta]$ and
${\teps}_i={\teps}[h_i;\delta]$, and for arbitrary $\tau\in[0,1]$
and $\pair{\eps}{\teps}\in{U}_\delta$ we write $h\left(\tau\right)
=\tau{h}_{2}+\at{1-\tau}h_1$,
as well as
\begin{align*}
\bar{g}_{i}\triple{\eps}{\teps}{\tau}=g_i[h\at{\tau};\eps,\teps,\delta],\quad
\eps\at\tau=\eps[h\at\tau;\delta],\quad
\teps\at\tau=\teps[h\at\tau;\delta],
\end{align*}
so that
\begin{align}
\label{Lem:FP.Prms.Cont.Eqn1}
\eps\at\tau=\bar{g}_{1}\triple{\eps\at\tau}{\teps\at\tau}{\tau},\quad
\teps\at\tau=\bar{g}_{2}\triple{\eps\at\tau}{\teps\at\tau}{\tau}
\end{align}
holds by construction. For fixed $\pair{\eps}{\teps}$ we estimate $\VarDer{\bar{g}_i}{\tau}$ as follows
\begin{align*}
\abs{\VarDer{\bar{g}_i}{\tau}\triple{\eps}{\teps}{\tau}}&\leq{}%
g_i[\abs{h_2-h_1};\eps,\teps,\delta]=
\eps^2\int\limits_0^\infty\xi\,a_\delta\at\xi\abs{h_2\at\xi-h_1\at\xi}
\Gamma_i\at{\xi;\eps,\teps,\delta}
\\&\leq\eps^2\snorm{h_2-h_1}\widehat{G}_i\at{1;\delta}
+ \eps^2\tnorm{h_2-h_1}\cnstKi,
\\&\leq{C}\eps^2\cnstKi\fnorm{h_2-h_1},
\end{align*}
compare \eqref{Lemma:Props.G.Func.Eqn1} and \eqref{Lemma:Props.G.Func.Eqn2}, and
Lemma \ref{Lem:Restr.Prms} combined with Lemma \ref{Lem:FP.Prms.Sol} provides%
\begin{align*}
\abs{\VarDer{g_i}{\tau}\triple{\eps\at\tau}{\teps\at\tau}{\tau}}
\leq{C}\eps\at{\tau}\fnorm{h_2-h_1}=C\at{\eps_2+\eps_1}\fnorm{h_2-h_1}.
\end{align*}
Differentiating \eqref{Lem:FP.Prms.Cont.Eqn1} with respect to $\tau$ yields
\begin{align*}
\eps^{\,\prime}=\VarDer{\bar{g}_1}{\eps}\eps^{\,\prime}+
\VarDer{\bar{g}_1}{\teps}{\teps}^{\,\prime}+ \VarDer{\bar{g}_1}{\tau},
\quad%
{\teps^{\,\prime}}=\VarDer{\bar{g}_2}{\eps}\eps^{\,\prime}+
\VarDer{\bar{g}_2}{\teps}{\teps}^{\,\prime}+ \VarDer{\bar{g}_2}{\tau},
\end{align*}
where $^{\prime}$ denotes $\frac{\d}{\d\tau}$.
Moreover, Lemma \ref{Lem:FP.Prms.Sol} combined with $\bar{g}_1=\eps$,
$\bar{g}_2=\teps$, and $\eps/\teps=\DO{1}$ provides
\begin{align*}
\VarDer{\bar{g}_1}{\eps}=\at{2+\Do{1}}\quad
\VarDer{\bar{g}_2}{\eps}=\DO{1},\quad \VarDer{\bar{g}_1}{\teps}=\Do{1},\quad
\VarDer{\bar{g}_2}{\teps}=\Do{1},
\end{align*}
and we conclude that
\begin{align*}
\left(%
\begin{array}{cc}
  -1+\Do{1} & \Do{1}  \\
  \DO{1} & 1+\Do{1}  \\
\end{array}%
\right)
\left(%
\begin{array}{c}
\eps^{\,\prime}\\{\teps}^{\,\prime}
\end{array}%
\right)
\sim{1}%
\left(%
\begin{array}{c}
\VarDer{\bar{g}_1}{\tau}\\\VarDer{\bar{g}_2}{\tau}
\end{array}%
\right).
\end{align*}
Finally, we find
\begin{align*}
\abs{\eps_{2}-\eps_{1}}+\abs{\teps_{2}-\teps_{1}}
=\int\limits_{0}^{1}\abs{\eps^{\,\prime}+{\teps}^{\,\prime}}d\tau
\leq{C}\at{\eps_1+\eps_2}\,\fnorm{h_2-h_1},
\end{align*}
which is the desired result.
\end{proof}
%
%
%-----------------------------------------------------------------
\subsection{Solving the fixed point equation for $\Phi$}\label{sect:FP.Phi}
%-----------------------------------------------------------------%
For each $h\in\fspace$ and arbitrary parameters
$\pair{\eps}{\teps}$ we define the function
\begin{align*}
J[h;\eps,\teps,\delta]\at{z}\deq\psi\at{z;\eps,\teps,\delta}\,
\int\limits_z^\infty\frac{\xi{a}_\delta\at{\xi}h\at\xi}{\psi\at{\xi;\eps,\teps,\delta}}\,d\xi,
\end{align*}
which is related to the fixed problem for $\Phi$ via
\begin{align*}
\bar{I}_\delta[\Phi]=\eps[\Phi*\Phi;\delta]\,J\big[\Phi*\Phi,\,\eps[\Phi*\Phi;\delta],\,\teps[\Phi*\Phi;\delta],\,\delta\big],
\end{align*}
compare \eqref{Def:MainFPOp}. Notice that the exponental decay of $h$ implies that the function $J[h;\eps,\teps,\delta]$ is well defined and has finite moments, i.e.,
\begin{align*}
\int\limits_0^\infty{}J[h;\eps,\teps,\delta]\at{z}\,dz<\infty,\qquad\int\limits_0^\infty{}zJ[h;\eps,\teps,\delta]\at{z}\,dz<\infty.
\end{align*}
Moreover, below we show, cf. Corollary \ref{Cor:PropsJres}, that the operator $J[\cdot;\eps,\teps;\delta]$ maps the space $\fspace$ into itself.
%%
%
%-----------------------------------------------------------------
\subsubsection{Approximation of the operator $J$}
%-----------------------------------------------------------------
%
In this section we show that the operator $J$ can be approximated by
\begin{align*}
J_\app[h;\eps,\teps,\delta]\at{z}\deq
\chi_{\ccinterval{0}{1}}\at{z}\,\psi\at{z;\eps,\teps,\delta}\,
\int\limits_0^\infty\xi\,J[h;\eps,\teps,\delta]\at\xi\,d\xi,%
\end{align*}
that means all contributions coming from
\begin{align*}
J_\res[h;\eps,\teps,\delta]\deq{}J[h;\eps,\teps,\delta]-J_\app[h;\eps,\teps,\delta]
\end{align*}
can be neglected. To prove this we split the operator $J$ into three
parts $J=J_1+J_2+J_3$ with
\begin{align}
\label{eqn:OpJ.Split}
\begin{split}
J_1[h;\eps,\teps,\delta]\at{z}&\deq
+\chi_{\ccinterval{0}{1}}\at{z}\,\psi\at{z;\eps,\teps,\delta}\,
\int\limits_0^{\infty}\frac{\xi{a}_\delta\at{\xi}h\at\xi}{\psi\at{\xi;\eps,\teps,\delta}}\,d\xi,
\\%
J_2[h;\eps,\teps,\delta]\at{z}&\deq -\chi_{\ccinterval{0}{1}}\at{z}
\,\displaystyle \psi\at{z;\eps,\teps,\delta}\,
\int\limits_0^{z}\frac{\xi{a}_\delta\at{\xi}h\at\xi}{\psi\at{\xi;\eps,\teps,\delta}}\,d\xi,
\\%
J_3[h;\eps,\teps,\delta]\at{z}&\deq
+\chi_{\cointerval{1}{\infty}}\at{z}\,\psi\at{z;\eps,\teps,\delta}\,
\int\limits_z^{\infty}\frac{\xi{a}_\delta\at{\xi}h\at\xi}{\psi\at{\xi;\eps,\teps,\delta}}\,d\xi.
\end{split}
\end{align}
Notice that $J_\app[h;\eps,\teps,\delta]$,
$J_1[h;\eps,\teps,\delta]$, and $J_2[h;\eps,\teps,\delta]$ are
supported in $\ccinterval{0}{1}$, whereas the support of
$J_3[h;\eps,\teps,\delta]$ equals $\cointerval{1}{\infty}$.
Moreover, the next result shows that $J_1$ does not contribute to
the residual operator $J_\res$.
\begin{remark}
\label{Rem:PropsJres} For all $h\in\fspace$ and all parameters
$\triple{\eps}{\teps}{\delta}$ we have
\begin{align*}
J_\res[h;\eps,\teps,\delta]\at{z}= J_2[h;\eps,\teps,\delta]\at{z}-
\psi\at{z;\eps,\teps,\delta}\at{\int\limits_0^1\xi\,J_2[h;\eps,\teps,\delta]\at{\xi}\,d\xi
+\int\limits_1^\infty\xi\,J_3[h;\eps,\teps,\delta]\at{\xi}\,d\xi},
\end{align*}
for $0\leq{z}\leq{1}$, as well as
\begin{align*}
J_\res[h;\eps,\teps,\delta]\at{z}=J_3[h;\eps,\teps,\delta]\at{z},
\end{align*}
for $z\geq{1}$.
\end{remark}
\begin{proof} The second assertion is a direct consequence
of \eqref{eqn:OpJ.Split}. Now let $0\leq{z}\leq{1}$.
Due to the normalization condition
$\int_0^1\xi\,\psi\at{\xi;\eps,\teps,\delta}\,d\xi=1$ we have
\begin{align*}
\int\limits_0^{\infty}z\,J_1[h;\eps,\teps,\delta]\at{z}\,dz =
\int\limits_0^{\infty}\frac{\xi{a}_\delta\at{\xi}h\at\xi}{\psi\at{\xi;\eps,\teps,\delta}}\,d\xi
\end{align*}
and this implies
\begin{align*}
\psi\at{z;\eps,\teps,\delta}
\int\limits_0^1\xi\,J_1[h;\eps,\teps,\delta]\at{\xi}\,d\xi
=J_1[h;\eps,\teps,\delta]\at{z}.
\end{align*}
Moreover, by definition we have
\begin{align*}
J_\res[h;\eps,\teps,\delta]\at{z}=
J_1[h;\eps,\teps,\delta]\at{z}+J_2[h;\eps,\teps,\delta]\at{z}-
\psi\at{z;\eps,\teps,\delta}\sum\limits_{i=1}^3\int\limits_0^\infty\xi\,J_i[h;\eps,\teps,\delta]\at{\xi}\,d\xi,
\end{align*}
and the combination of both results yields the first assertion.
\end{proof}
In the next step we estimate the operators $J_2$ and $J_3$ as well
as their derivatives with respect to $\pair{\eps}{\teps}$.
\begin{lemma}
\label{Lem:PropsJres} For all $h\in\fspace$ and all parameters
$\triple{\eps}{\teps}{\delta}$ that satisfy Assumption
\ref{Ass:Prms.0} we find
\begin{align*}
\babs{J_{2}[h;\eps,\teps,\delta]\at{z}}+
\babs{\VarDer{}{\eps}J_{2}[h;\eps,\teps,\delta]\at{z}}+
\babs{\VarDer{}{\teps}J_{2}[h;\eps,\teps,\delta]\at{z}}
\leq%
\frac{C}{\delta}\snorm{h}
\end{align*}
for all $0\leq{z}\leq{1}$, as well as
\begin{align*}
\babs{J_{3}[h;\eps,\teps,\delta]\at{z}}+
\babs{\VarDer{}{\eps}J_{3}[h;\eps,\teps,\delta]\at{z}}+
\babs{\VarDer{}{\teps}J_{3}[h;\eps,\teps,\delta]\at{z}}
\leq%
{C}\tnorm{h}\frac{\exp\at{-\beta_1{z}}}{z^{\beta_2}}
\end{align*}
for all ${z}\geq{1}$.
\end{lemma}
\begin{proof}
Let $0\leq{z}\leq{1}$. The definition of $J_2$ provides
\begin{align*}
J_{2}[h;\eps,\teps,\delta]\at{z}=-
\int\limits_0^z\xi{a}_\delta\at{\xi}h\at\xi \exp\at{\int\limits_z^
\xi{a}_\delta\at{y}\at{b_\delta\at{y}-\eps-\teps{y}}\,dy}\,d\xi, %
\end{align*}
and we estimate
\begin{align*}
\babs{J_{2}[h;\eps,\teps,\delta]\at{z}}&\leq%
\snorm{h}\exp\at{C\frac{\eps+\teps}{\sqrt\delta}}
\int\limits_0^z\xi{a}_\delta\at{\xi}
\exp\at{-\int\limits_\xi^z{a}_\delta\at{y}
\min\{0,b_\delta\at{y}\}\,dy}\,d\xi %
\\&\leq%
\snorm{h}C\int\limits_0^z\xi{a}_\delta\at{\xi}
\exp\at{-\int\limits_0^1{a}_\delta\at{y}
\min\{0,b_\delta\at{y}\}\,dy}\,d\xi %
\\&\leq%
\snorm{h}C\int\limits_0^1\xi{a}_\delta\at{\xi}
\,d\xi
\leq%
\frac{C}{\sqrt\delta}\snorm{h}.
\end{align*}
Moreover,
\begin{align*}
\babs{\VarDer{}{\eps}J_{2}[h;\eps,\teps,\delta]\at{z}}&\leq%
\abs{ \int\limits_0^z\xi{a}_\delta\at{\xi}h\at\xi
\exp\at{\int\limits_z^\xi{a}_\delta\at{y}\at{b_\delta\at{y}-\eps-\teps{y}}\,dy}
\at{\int\limits_z^\xi{a}_\delta\at{y}\,dy}\,d\xi}&
\\&
\leq%
\frac{1}{\sqrt\delta}J_{2}\big[\abs{h};\eps,\teps,\delta\big]\at{z}
\leq%
\frac{C}{\delta}\snorm{h},
\end{align*}%
and the estimate for
$\VarDer{}{\teps}J_{2}[h;\eps,\teps,\delta]\at{z}$ is entirely
similar. Now let $z\geq{1}$. Then,
\begin{align*}
\babs{J_{3}[h;\eps,\teps,\delta]\at{z}}
&\leq%
\int\limits_z^\infty\xi{a}_\delta\at{\xi}
\abs{h\at\xi}\exp\at{\int\limits_z^\xi{a}_\delta\at{y}
\at{b_\delta\at{y}-\eps-\teps{y}}\,dy}\,d\xi %
\\&\leq%
C\int\limits_z^\infty
\abs{h\at\xi}\exp\at{\int\limits_z^\xi{a}_\delta
\at{y}\,b_\delta\at{y}\,dy}\,d\xi %
\\&\leq%
C\int\limits_z^\infty
\abs{h\at\xi}\exp\at{C+3\ln\xi-3\ln{z}}\,d\xi %
=\frac{C}{z^3}\int\limits_z^\infty
\xi^3\abs{h\at\xi}\,d\xi, %
\end{align*}
as well as
\begin{align*}
\babs{\VarDer{}{\teps}J_{3}[h;\eps,\teps,\delta]\at{z}}&+
\babs{\VarDer{}{\teps}J_{3}[h;\eps,\teps,\delta]\at{z}}\leq
\\&\leq%
\int\limits_z^\infty\xi{a}_\delta\at{\xi}
\abs{h\at\xi}\exp\at{\int\limits_z^\xi{a}_\delta
\at{y}\at{b_\delta\at{y}-\eps-\teps{y}}\,dy}
\at{\int\limits_z^\xi{a}_\delta\at{y}\at{y+1}\,dy}\,d\xi %
\\&\leq%
C\int\limits_z^\infty
\abs{h\at\xi}\at{\xi-z}\exp\at{C+3\ln\xi-3\ln{z}}\,d\xi %
=\frac{C}{z^3}\int\limits_z^\infty
\xi^3\at{\xi-z}\abs{h\at\xi}\,d\xi.
\end{align*}
Finally, using \eqref{Lem:Props.H.1.Eqn3} completes the proof.
\end{proof}
As a consequence of Remark \ref{Rem:PropsJres} and Lemma
\ref{Lem:PropsJres} we obtain estimates for the residual operator. In particular, $h\in\fspace$ implies $J_\res[h;\eps,\teps,\delta]\in\fspace$, and hence $J_\app[h;\eps,\teps,\delta]\in\fspace$.
\begin{corollary}
\label{Cor:PropsJres} The assumptions from Lemma \ref{Lem:PropsJres}
imply
\begin{align*}
J_\res[h;\eps,\teps,\delta]\in\fspace,\quad
\VarDer{}{\eps}J_\res[h;\eps,\teps,\delta]\in\fspace,\quad
\VarDer{}{\teps}J_\res[h;\eps,\teps,\delta]\in\fspace
\end{align*}
with
\begin{align*}
\snorm{J_\res[h;\eps,\teps,\delta]}\leq{}\frac{C}{\delta}\snorm{h}+C\tnorm{h},\quad\quad
\tnorm{J_\res[h;\eps,\teps,\delta]}\leq{C}\tnorm{h}
\end{align*}
as well as
\begin{align*}
\fnorm{\VarDer{}{\eps}{J_\res[h;\eps,\teps,\delta]}}+
\fnorm{\VarDer{}{\teps}{J_\res[h;\eps,\teps,\delta]}}\leq\frac{C}{\delta}\fnorm{h}.
\end{align*}
In particular, for fixed $h$ we have
\begin{align*}
\fnorm{J_\res[h;\eps_2,\teps_2,\delta]-J_\res[h;\eps_1,\teps_1,\delta]}\leq\frac{C}{\delta}\fnorm{h}
\at{\abs{\eps_2-\eps_1}+\abs{\teps_2-\teps_1}}.
\end{align*}
\end{corollary}
\begin{proof}
All assertions are direct consequences of Remark \ref{Rem:PropsJres}
and Lemma \ref{Lem:PropsJres}.
\end{proof}
%
%-----------------------------------------------------------------
\subsubsection{Setting for the fixed point problem}
%-----------------------------------------------------------------
%
%
Here we introduce suitable subsets of the function space $\fspace$
that allow to apply the contraction principle for the operator
$\bar{I}_\delta$. For this reason we introduce the functions
\begin{align*}
\widehat{\Phi}_\delta\at{z}\deq\chi_{\ccinterval{0}{1}}\at{z}\,\psi\at{z;0,0,\delta},
\end{align*}
which satisfy $\fnorm{\widehat{\Phi}_\delta-\Phi_{LSW}}\to0$ as
$\delta\to0$, see \eqref{eq:Def.HomSol} and \eqref{eq:Mass.Norm}.
\begin{definition}
Let $\mu_\delta$ be a number with
\begin{align}
\mu_\delta=\Do{1},\qquad\frac{1}{\delta\,\cnstKO}=\Do{\mu_\delta^2},\qquad
\fnorm{\widehat{\Phi}_\delta-\Phi_{LSW}}=\Do{\mu_\delta^2},
\label{Def:PrmMu}
\end{align}
and define the sets
\begin{align*}
\calY_\delta=\{\Phi\in\fspace\;:\;\fnorm{\Phi-\widehat{\Phi}_\delta}\leq\mu_\delta^2\},\quad
\calZ_\delta=\{h\in\fspace\;:\;\fnorm{h-\widehat{\Phi}_\delta*\widehat{\Phi}_\delta}\leq\mu_\delta\}.
\end{align*}
\end{definition}
\begin{lemma}
\label{Lem:Props.Spaces}%
For all sufficiently small $\delta$ the following
assertions are satisfied.
\begin{enumerate}
\item $\Phi_{LSW}\in\calY_\delta$,
\item $\Phi\in\calY_\delta$ implies $\Phi*\Phi\in\calZ_\delta$,
\item each $h\in\calZ_\delta$ satisfies Assumption \ref{Ass:Prms}, hence there exist the solutions
$\eps[h;\delta]$ and $\teps[h;\delta]$ from Lemma
\ref{Lem:FP.Prms.Sol}.
\end{enumerate}
\end{lemma}
\begin{proof}
The first assertion holds by construction. For $\Phi\in\calY_\delta$
we have
\begin{align*}
\fnorm{\Phi*\Phi-\widehat{\Phi}_\delta*\widehat{\Phi}_\delta}
&=%
\fnorm{2\Phi*\nat{\Phi-\widehat{\Phi}_\delta} + \nat{\Phi-\widehat{\Phi}_\delta}*\nat{\Phi-\widehat{\Phi}_\delta}}
\\&\leq%
2\fnorm{\Phi*\nat{\Phi-\widehat{\Phi}_\delta}}+
\fnorm{\nat{\Phi-\widehat{\Phi}_\delta}*\nat{\Phi-\widehat{\Phi}_\delta}}
\leq{}C\fnorm{\Phi-\widehat{\Phi}_\delta}=\Do{\mu_\delta}.
\end{align*}%
This implies
\begin{math}
\fnorm{\widehat{\Phi}_\delta*\widehat{\Phi}_\delta-\Phi_{LSW}*\Phi_{LSW}}
\leq{C}%
\fnorm{\widehat{\Phi}_\delta-\Phi_{LSW}}\leq{C}{\mu_\delta}^2=\Do{\mu_\delta}
\end{math} %
, and for all $h\in\calZ_\delta$ we find
\begin{align*}
\fnorm{h-\Phi_{LSW}*\Phi_{LSW}}\leq\fnorm{h-\widehat{\Phi}_\delta*\widehat{\Phi}_\delta}+
C\fnorm{\widehat{\Phi}_\delta-\Phi_{LSW}}=\Do{\mu_\delta}=\Do{1}.
\end{align*}
Therefore,
\begin{align*}
\tnorm{h}=\tnorm{\Phi_{LSW}*\Phi_{LSW}}+\Do{\mu_\delta}=\Do{1},
\end{align*}
and the proof is complete.
\end{proof}
%
%
%
%-----------------------------------------------------------------
\subsubsection{Contraction principle for $\Phi$ }
%-----------------------------------------------------------------
%
Recall that the solution $\pair{\eps}{\teps}[h;\delta]$ from Lemma
\ref{Lem:FP.Prms.Sol} satisfies
\begin{align*}
\frac{1}{\eps[h;\delta]}=
\int\limits_0^\infty\xi\,J[h;\eps[h;\delta],\teps[h;\delta],\delta].
\end{align*}
Therefore we define $I[h;\delta]=I_\app[h;\delta]+I_\res[h;\delta]$
with
\begin{align*}
I_\app[h;\delta]\at{z}
\deq
\eps[h;\delta]\,J_\app[h;\eps[h;\delta],\teps[h;\delta],\delta],\quad\quad
I_\res[h;\delta]\at{z}
\deq
\eps[h;\delta]\,J_\res[h;\eps[h;\delta],\teps[h;\delta],\delta],\quad\quad
\end{align*}
and this implies $\bar{I}_\delta[\Phi]=I[\Phi*\Phi;\delta]$, with
$\bar{I}_\delta$ as in \eqref{Def:MainFPOp}, as well as
\begin{align}
\notag%\label{Eqn:IappId}%
 I_\app[h;\delta]\at{z}=\chi_{\ccinterval{0}{1}}\at{z}\,\psi
\at{z;\eps[h;\delta],\teps[h;\delta],\delta}.
\end{align}
In particular, is close to $\widehat{\Phi}_\delta$ with error controlled by $\eps$ and $\teps$, provided that  $\delta$ is small and
$h$ close to $\Phi_{LSW}*\Phi_{LSW}$.
\begin{lemma}
\label{Lem:OpI}
For sufficiently small $\delta$ the operator $I$ maps $\calZ_\delta$
into $\calY_\delta$, and is Lipschitz continuous with arbitrary
small constant. More precisely,
\begin{align*}
\fnorm{I[h_2;\delta]-I[h_1;\delta]}\leq\Do{1}\, \fnorm{h_2-h_1}
\end{align*}
for all $h_1,h_2\in\calZ_\delta$.
\end{lemma}
\begin{proof}
For each $h\in\calZ_\delta$ Lemma \ref{Lem:FP.Prms.Cont} provides
\begin{align*}
\fnorm{I_\res[h;\delta]}\leq{C}\frac{\eps[h;\delta]}{\delta}\fnorm{h}\leq
\frac{C}{\delta\,\cnstKO}\fnorm{h},
\end{align*}
where we used Lemma \ref{Lem:Restr.Prms} and Lemma
\ref{Lem:FP.Prms.Sol}. Moreover, \eqref{Def:PrmMu} implies
\begin{align*}
\fnorm{I_\res[h;\delta]}\leq\Do{\mu_\delta^2}\fnorm{h}=\Do{\mu_\delta^2}\fnorm{\widehat{\Phi}_\delta*\widehat{\Phi}_\delta}=\Do{\mu_\delta^2}.
\end{align*}
From Lemma \ref{Lem:Prop.HomSol} we derive
\begin{align*}
\widetilde{C}^{-1}\,\widehat{\Phi}_\delta\at{z}
\leq{}\psi\at{z;\eps[h;\delta],\teps[h;\delta],\delta}
\leq\widetilde{C}\,\widehat{\Phi}_\delta\at{z},
\end{align*}
for all $0\leq{z}\leq{1}$ with
\begin{align*}
\widetilde{C}[h;\delta]={\exp\at{C\frac{\eps[h;\delta]+\teps[h;\delta]}{\sqrt\delta}}}=\exp\at{\Do{\mu_\delta^2}}.
\end{align*}
We conclude
\begin{align*}
\babs{I_\app[h;\delta]\at{z}-\widehat{\Phi}_\delta\at{z}}=
\at{\widetilde{C}[h;\delta]-1}\widehat{\Phi}_\delta\at{z}=\Do{\mu_\delta^2},
\end{align*}
and find
$\fnorm{I[h;\delta]-\widehat{\Phi}_\delta}=\Do{\mu_\delta^2}$,
which implies $I[h;\delta]\in\calY_\delta$ for small $\delta$. The
Lipschitz continuity of $I_\res$ follows from Lemma
\ref{Lem:FP.Prms.Cont} and Corollary \ref{Cor:PropsJres}, and the
Lipschitz continuity of $I_\app$ is a consequence of Lemma
\ref{Lem:Prop.HomSol}. Moreover, using \eqref{Def:PrmMu} and the
same estimates as above we find that the Lipschitz constants are of
order $\Do{\mu_\delta^2}$.
\end{proof}
Now we can prove Theorem \ref{MainTheo2} and Proposition \ref{Intro.Prop2} from \S \ref{sec:Problem}.
\begin{corollary} The operator $\bar{I}_\delta$ is a contraction of $\calY_\delta$, and thus
there exists a unique fixed point in $\calY_\delta$. Moreover, this
fixed point is nonnegative since the cone of nonnegative function is
invariant under the action of $\bar{I}_\delta$.
\end{corollary}
\begin{proof}
By construction, we have $\bar{I}_\delta[\Phi]=I[\Phi*\Phi;\delta]$
and all assertions follow from Lemma \ref{Lem:Props.Spaces} and
Lemma \ref{Lem:OpI}.
\end{proof}
Finally, we discuss the influence of the parameters
$\pair{\beta_1}{\beta_2}$ that control the decay behavior of the
solution, see \eqref{Def:Spaces}.
\begin{remark}
\label{Rem.Prms.Beta} Consider another pair of parameters
$\npair{\tilde{\beta}_1}{\tilde{\beta}_2}$ with
$\tilde{\beta}_1>\beta_1$ and $\tilde{\beta}_2>\beta_2$, and denote
by $\widetilde{\calY}_\delta$ the corresponding set from Lemma
\ref{Lem:Props.Spaces}. Our previous results imply, compare Remark
\ref{Rem:Props.H.2} and Remark \ref{Ass:Prms.Rem}, that
$\widetilde{\calY}_\delta\subset{\calY}_\delta$ for all small
$\delta$, and this yields the following two assertions. $\at{i}$ For
small but fixed $\delta$ the solution that is found with the
parameters $\npair{\tilde{\beta}_1}{\tilde{\beta}_2}$ equals the
solution for $\pair{\beta_1}{\beta_2}$. $\at{ii}$ The smaller
$\delta$ is the larger we can choose the parameters
$\pair{\beta_1}{\beta_2}$, i.e., the faster the solution decays.
\end{remark}%
\textbf{Acknowledgements.} MH and BN gratefully acknowledge support
through the DFG Research Center {\sc Matheon } and the DFG Research
group {\it Analysis and Stochastics in complex physical systems}.
JJLV was supported through the Max Planck Institute for Mathematics
in the Sciences, the Alexander-von-Humboldt foundation, and DGES
Grant MTM2007-61755.
\end{document}